\documentclass[reqno]{amsart}
\numberwithin{equation}{section}

\usepackage{hyperref}
\usepackage{bm}
\usepackage{amssymb}
\usepackage{amsthm}
\usepackage[alphabetic,nobysame]{amsrefs}
\usepackage{graphicx} 
\usepackage{amscd}
\usepackage{enumerate}
\usepackage{cancel}
\usepackage[font=footnotesize]{caption}
\usepackage{dsfont}
\usepackage[colorinlistoftodos]{todonotes}
\usepackage{tikz-cd}
\usepackage[all]{xy}
\usepackage{mathtools}
\usepackage{scrextend}

\definecolor{darkred}{rgb}{1,0,0} 
\definecolor{darkgreen}{rgb}{0,0.8,0}
\definecolor{darkblue}{rgb}{0,0,1}

\hypersetup{colorlinks,
linkcolor=darkblue,
filecolor=darkgreen,
urlcolor=darkblue,
citecolor=darkgreen}

\makeatletter
\providecommand\@dotsep{5}
\makeatother

\definecolor{orange}{RGB}{253,85,0}
\definecolor{darkgreen}{RGB}{0,95,10}

 \newcommand{\trap}{\mathrm{trap}}

 \newcommand{\inj}{\mathrm{inj}}
 \newcommand{\YY}{\mathcal{Y}}
 \newcommand{\xx}{\bm{x}}
 \newcommand{\yy}{\bm{y}}
 \newcommand{\vv}{\bm{v}}
 
 \newcommand{\Eu}{E^u}
 
 \newcommand{\Z}{\mathds{Z}}
 
 \newcommand{\R}{\mathds{R}}
 \newcommand{\PP}{\mathds{P}}

 \newcommand{\XX}{\mathcal{X}}
 \newcommand{\UU}{\mathcal{U}}
 \newcommand{\VV}{\mathcal{V}}
 \newcommand{\WW}{\mathcal{W}}
 \newcommand{\ZZ}{\mathcal{Z}}
 \newcommand{\LL}{\mathcal{L}}
 \newcommand{\KK}{\mathcal{K}}
 \newcommand{\GG}{\mathcal{G}}

 \newcommand{\HH}{\mathcal{H}}

 \newcommand{\Diff}{\mathrm{Diff}}
 \newcommand{\area}{\mathrm{area}}
 \newcommand{\sigmav}{\sigma^{\mathrm{v}}}
 \newcommand{\sigmavv}[1]{\sigma^{\mathrm{v}}_{#1}}
 \newcommand{\crit}{\mathrm{crit}}
 
 \newcommand{\id}{\mathrm{id}}
 
 \newcommand{\ind}{\mathrm{ind}}
 \newcommand{\nul}{\mathrm{nul}}
 \newcommand{\Imm}{\mathrm{Imm}}

 \DeclareMathOperator{\interior}{int}
  
 \DeclareMathOperator{\rank}{rank} 
 \DeclareMathOperator*{\toup}{\longrightarrow} 
 \DeclareMathOperator*{\ttoup}{\llongrightarrow} 
  
 \DeclareMathOperator*{\eembup}{\llonghookrightarrow}

\DeclareRobustCommand{\llonghookrightarrow}{\lhook\joinrel\relbar\joinrel\relbar\joinrel\rightarrow}
\DeclareRobustCommand{\llongrightarrow}{\relbar\joinrel\relbar\joinrel\rightarrow}

\newtheoremstyle{personal}
{12pt}
{12pt}
{\itshape}
{}
{\bfseries}
{.}
{.5em}
{}%

\newtheoremstyle{break}
{12pt}
{12pt}
{\itshape}
{}
{\bfseries}
{}
{.5em}
{}%

 \theoremstyle{personal}
 \newtheorem{MainThm}{Theorem}

 \newtheorem{Thm}{Theorem}[section]
 \newtheorem{Prop}[Thm]{Proposition}
 \newtheorem{Lemma}[Thm]{Lemma}
 \newtheorem{Cor}[Thm]{Corollary}

 \theoremstyle{break}

 \newtheorem{LemmaB}[Thm]{Lemma}

\newtheoremstyle{newdefinition}
{10pt}
{10pt}
{}
{}
{\bfseries}
{.}
{.5em}
{}%
 
 \theoremstyle{newdefinition}
 
 \newtheorem{Definition}[Thm]{Definition}
 \newtheorem{Remark}[Thm]{Remark}

\title[From curve shortening to link stability and Birkhoff sections]{From curve shortening to flat link stability\\ and Birkhoff sections of geodesic flows}

\author[M. R. R. Alves]{Marcelo R. R. Alves}
\address{Marcelo R.R. Alves\newline\indent Department of Mathematics, University of Antwerp\newline\indent Middelheim G, M.G.105, Middelheimlaan 1, 2020 Antwerp, Belgium}
\email{marcelo.ribeiroderesendealves@uantwerpen.be}

\author[M. Mazzucchelli]{Marco Mazzucchelli}
\address{Marco Mazzucchelli\newline\indent CNRS, UMPA, \'Ecole Normale Sup\'erieure de Lyon\newline\indent 46 all\'ee d'Italie, 69364 Lyon, France}
\email{marco.mazzucchelli@ens-lyon.fr}

\thanks{Marcelo R.~R.~Alves was supported by the Senior Postdoctoral fellowship of the Research Foundation - Flanders (FWO) in fundamental research 1286921N. Marco Mazzucchelli is partially supported by the ANR grants CoSyDy (ANRCE40-0014) and COSY (ANR-21-CE40-0002).}

\date{August 21, 2024. \emph{Revised}: October 8, 2024.}



\begin{document}

\begin{abstract}
We employ the curve shortening flow to establish three new results on the dynamics of geodesic flows of closed Riemannian surfaces. The first one is the stability, under $C^0$-small perturbations of the Riemannian metric, of certain flat links of closed geodesics. The second one is a forced existence theorem for closed connected orientable Riemannian surfaces: for surfaces of positive genus, the existence of a contractible simple closed geodesic $\gamma$ forces the existence of infinitely many closed geodesics intersecting $\gamma$ in every primitive free homotopy class of loops; for the 2-sphere, the existence of two disjoint simple closed geodesics forces the existence of a third one intersecting both. The final result asserts the existence of Birkhoff sections for the geodesic flow of any closed connected orientable Riemannian surface.
\end{abstract}

\maketitle

\vspace{-20pt}

\begin{footnotesize}
\begin{quote}
\tableofcontents 
\end{quote}
\end{footnotesize}

\section{Introduction}

On a closed Riemannian surface, the curve shortening flow is the $L^2$ anti-gradient flow of the length functional on the space of immersed loops. Unlike other more conventional anti-gradient flows on loop spaces, such as the one of the energy functional in the $W^{1,2}$ settings \cite{Klingenberg:1978aa}, the curve shortening flow is only a semi-flow (i.e.~its orbits are only defined in positive time), and the very existence of its trajectories in long time was a remarkable theorem of geometric analysis, first investigated by Gage \cite{Gage:1983aa,Gage:1984aa, Gage:1990aa} and Hamilton \cite{Gage:1986aa}, fully settled for embedded loops by Grayson \cite{Grayson:1989aa}, and further generalized to immersed loops by Angenent \cite{Angenent:1990aa, Angenent:1991aa}. One of the remarkable properties of the curve shortening flow is that it shrinks loops without increasing the number of their self-intersections. This allowed Grayson to provide a rigorous proof of Lusternik-Schnirelmann's theorem on the existence of three simple closed geodesics on every Riemannian 2-sphere \cite{Grayson:1989aa, Mazzucchelli:2018aa}. Later on, Angenent \cite{Angenent:2005aa} framed the curve shortening flow in the setting of Morse-Conley theory \cite{Conley:1978aa}, and proved a spectacular existence result for closed geodesics of certain prescribed flat-knot types on closed Riemannian surfaces. 

The purpose of this article, which is inspired by this latter work of Angenent, is to present new applications of the curve shortening flow to the study of the dynamics of geodesic flows: the stability of certain configurations of closed geodesics under $C^0$ perturbation of the Riemannian metric, the forced existence of closed geodesics intersecting certain given ones, and the existence of Birkhoff sections. We present our main results in detail over the next three subsections.

\subsection{$C^0$-stability of flat links of closed geodesics}\label{ss:C0_stability}

The expression of a Riemannian geodesic vector field involves the first derivatives of the Riemannian metric. Therefore, for each integer $k\geq1$, a $C^k$-small perturbation of the Riemannian metric corresponds to a $C^{k-1}$-small perturbation of the geodesic vector field. However, a $C^0$-small perturbation of the Riemannian metric may result in a drastic deformation of the geodesic vector field and of its dynamics. For instance, given any smooth embedded circle $\gamma$ in a Riemannian surface, one can always find a $C^0$-small perturbation of the Riemannian metric that makes $\gamma$ a closed geodesic for the new metric \cite[Ex.~43]{Alves:2022aa}. Moreover, it is always possible to arbitrarily increase the topological entropy of the geodesic flow by means of a $C^0$-small perturbation of the Riemannian metric \cite[Th.~12]{Alves:2022aa}. 

From the geometric perspective, it is natural to consider the $C^0$ topology on the space of Riemannian metrics: indeed, the length of curves, or more generally the volume of compact submanifolds, vary continuously under $C^0$-deformations of the Riemannian metric.
A result of the first author, Dahinden, Meiwes and Pirnapasov \cite{Alves:2023aa} asserts that the topological entropy of a non-degenerate geodesic flow of a closed Riemannian surface cannot be destroyed by a $C^0$-small perturbation of the metric. In a nutshell, this can be expressed by saying that the chaos of such geodesic flows is $C^0$ robust.
Our first result provides another geometric dynamical property that, unexpectedly, survives after $C^0$-perturbations of the Riemannian metric of a closed orientable surface: the existence of suitable configurations of closed geodesics. 
In order to state the result precisely, let us first introduce the setting.

Let $(M,g)$ be a closed Riemannian surface. We denote by $\Imm(S^1,M)$ the space of smooth immersions of the circle $S^1=\R/\Z$ to $M$, endowed with the $C^3$ topology. The group of orientation preserving smooth diffeomorphisms $\Diff_+(S^1)$ acts on $\Imm(S^1,M)$ by reparametrization, and we denote the quotient by
\[\Omega:=\frac{\Imm(S^1,M)}{\Diff_+(S^1)}.\]
The space $\Omega$ consists of unparametrized oriented immersed loops in $M$, and is endowed with the quotient $C^3$ topology. 
The length functional 
\[
L_g:\Omega\to(0,\infty),\qquad L_g(\gamma)=\int_{S^1} \|\dot\gamma(t)\|_g\,dt
\]
is well defined on $\Omega$, meaning that $L_g(\gamma)$ is independent of the specific choice of representative of $\gamma$ in $\Imm(S^1,M)$, is continuous, and even differentiable for a suitable differentiable structure on $\Omega$.
For each integer $n\geq1$, we denote by $\Delta_n$ the closed subset of $\Omega^{\times n}=\Omega\times...\times\Omega$ consisting of those multi-loops $\bm\gamma=(\gamma_1,...,\gamma_n)$ such that $\gamma_i$ is tangent to $\gamma_j$ for some $i\neq j$, or $\gamma_i$ has a self-tangency for some $i$. 
A path-connected component $\LL$ of $\Omega^{\times n}\setminus\Delta_n$ is called a \emph{flat link type}. The reason for this terminology is that $\LL$ lifts into a connected component of the space of links in the projectivized tangent bundle $\PP TM$. If $n=1$, $\LL$ is more specifically called a \emph{flat knot type}. This terminology was introduced by Arnold in \cite{Arnold:1994aa}.

For each integer $m\geq2$, we denote by $\gamma^m\in\Omega$ the $m$-fold iterate of a loop $\gamma\in\Omega$. Namely, once we fix a parametrization $\gamma:S^1\looparrowright M$, we obtain the parametrization $\gamma^m:S^1\looparrowright M$, $\gamma^m(t)=\gamma(mt)$. 
A loop $\gamma\in\Omega$ is \emph{primitive} if it is not of the form $\gamma=\zeta^m$ for some $\zeta\in\Omega$ and $m\geq2$, and otherwise it is an \emph{iterated} loop. 
A whole flat knot type $\KK$ is primitive when its closure in $\Omega$ contains only primitive loops. Examples of primitive flat knot types include all flat knot types consisting of embedded loops, and all flat knot types consisting of loops whose integral homology class is primitive (i.e.~not a multiple $mh$, for $m\geq2$, of another homology class $h$). Any flat link type $\LL$ is contained in a product $\KK_1\times...\times\KK_n$, where the factors $\KK_i$ are flat knot types, and we say that $\LL$ is primitive when all the factors $\KK_i$ are primitive flat knot types.

Closed geodesics admit a variational characterization and a dynamical one. The unparametrized oriented closed geodesics of $(M,g)$ are the critical point of the length functional $L_g$. The closed geodesics parametrized with unit speed are the base projections of the periodic orbits of the geodesic flow on the unit tangent bundle $\psi_t:SM\to SM$, $\psi_t(\dot\gamma(0))=\dot\gamma(t)$; here, $\gamma:\R\to M$ is any geodesic parametrized with unit speed $\|\dot\gamma\|_g\equiv1$. A closed geodesic $\gamma$ of length $\ell$ is non-degenerate when its unit-speed lift $\dot\gamma$ is a non-degenerate $\ell$-periodic orbit of the geodesic flow, meaning that $\dim\ker(d\psi_\ell(\dot\gamma(0))-\id)=1$. We introduce the following notion.

\begin{Definition}\label{d:stable_link}
A flat link of closed geodesics $\bm\gamma=(\gamma_1,...,\gamma_n)$ is \emph{stable} when every component $\gamma_i$ is non-degenerate and, for each $i\neq j$, the components $\gamma_i,\gamma_j$ have distinct flat knot types or distinct lengths $L_g(\gamma_i)\neq L_g(\gamma_j)$.\end{Definition}

Our first main result is the following.

\begin{MainThm}\label{mt:links}
Let $(M,g)$ be a closed Riemannian surface, $\LL$ a primitive flat link type, and $\bm\gamma\in\LL$ a stable flat link of closed geodesics. For each $\epsilon>0$, any Riemannian metric $h$ sufficiently $C^0$-close to $g$ has a flat link of closed geodesics $\bm\zeta\in\LL$ such that $\|L_h(\bm\zeta)- L_g(\bm\gamma)\|<\epsilon$.
\end{MainThm}

Our inspiration for Theorem \ref{mt:links} comes from Hofer geometry \cite{Hofer:1990aa, Polterovich:2001aa}. The Hamiltonian diffeomorphisms group of a symplectic manifold $\mathrm{Ham}(W,\omega)$ admits a remarkable metric, called the Hofer metric, which has a $C^0$ flavor and plays an important role in Hamiltonian dynamics and symplectic topology. When the symplectic manifold $(W,\omega)$ is a closed surface, 
a finite collection of 1-periodic orbits of a non-degenerate Hamiltonian diffeomorphism $\phi\in\mathrm{Ham}(W,\omega)$ has a certain braid type $\mathcal B$. The first author and Meiwes \cite{Alves:2021aa} proved that this property is stable under perturbation that are small with respect to the Hofer metric: any other $\psi\in\mathrm{Ham}(W,\omega)$ that is sufficiently close to $\phi$ also has a collection of 1-periodic orbits of braid type $\mathcal B$. The proof of this result involves Floer theory and holomorphic curves. Later on, employing periodic Floer homology, Hutchings \cite{Hutchings:2023aa} generalized the result to finite collections of periodic orbits of arbitrary period. 
Our Theorem \ref{mt:links} can be seen as a Riemannian version of these results. Unlike \cite{Alves:2021aa, Hutchings:2023aa}, our proof does not need Floer theory nor holomorphic curves, and instead employs the curve shortening flow.

\subsection{Forced existence of closed geodesics}

Our next result is an instance of a forcing phenomenon in dynamics: the existence of a particular kind of periodic orbit implies certain unexpected dynamical consequences. For instance, the existence of a hyperbolic periodic point with a transverse homoclinic for a diffeomorphism implies the existence of a horseshoe, which in turn implies the existence of plenty of nearby periodic points and the positivity of topological entropy \cite[Th.~6.5.5]{Katok:1995aa}. More in the spirit of our article, Boyland \cite{Boyland:1994aa} proved that the existence of periodic orbits with complicated braid types for a surface diffeomorphism implies  complicated dynamical structure, such as the positivity of the topological entropy and the existence of periodic orbits of certain other braid types.

In the specific case of geodesic flows, Denvir and Mackay \cite{Denvir:1998aa} proved that the existence of a contractible closed geodesic $\gamma$ on a Riemannian torus, or of three simple closed geodesics $\gamma_1,\gamma_2,\gamma_3$ bounding disjoint disks on a Riemannian 2-sphere, forces the positivity of the topological entropy of the corresponding geodesic flows. A simple argument involving the curve shortening flow further implies the existence of infinitely many closed geodesics in the complement of $\gamma$ or of $\gamma_1\cup\gamma_2\cup\gamma_3$. The forcing theory of Denvir and Mackay was generalized to the category of Reeb flows in \cite{Alves:2022ab,Pirnapasov:2021}.

Our second main result is a forced existence theorem for closed geodesics intersecting a simple one, on closed orientable surfaces of positive genus. The statement employs the following standard terminology.
A free homotopy class of loops in a surface $M$ is a connected component of the free loop space $C^\infty(S^1,M)$. Notice that this notion is less specific than the one of flat knot type: while every flat knot type corresponds to a unique free homotopy class of loops, a free homotopy class of loops always corresponds to infinitely many flat knot types. 
A free homotopy class of loops is called \emph{primitive} when it does not contain iterated loops.

\begin{MainThm}\label{mt:multiplicity}
On any closed connected orientable Riemannian surface of positive genus with a contractible simple closed geodesic $\gamma$, every primitive free homotopy class of loops contains infinitely many closed geodesics intersecting $\gamma$.
\end{MainThm}

For the 2-sphere, which is not covered by Theorem~\ref{mt:multiplicity}, we prove the following.

\begin{MainThm}\label{mt:sphere}
On any Riemannian 2-sphere with two disjoint simple closed geodesics $\gamma_1$ and $\gamma_2$, there exists a simple closed geodesic $\gamma$ intersecting each $\gamma_i$ in exactly two points, i.e.
$\# (\gamma\cap\gamma_1) = \#(\gamma\cap\gamma_2)=2$.
\end{MainThm}

The idea of the proof of Theorem~\ref{mt:multiplicity} is to employ a specific Riemannian metric due to Donnay, Burns, and Gerber \cite{Donnay:1988ab, Burns:1989aa}, which has $\gamma$ as closed geodesic. The properties of such a Riemannian metric will allow us to establish the non-vanishing of the local homology of infinitely many ``relative'' flat knot types consisting of loops intersecting $\gamma$. Since the local homology of a flat knot type $\KK$ relative to $\gamma$ is independent of the choice of the Riemannian metric having $\gamma$ as closed geodesic, and its non-vanishing implies the existence of at least one closed geodesic of relative flat knot type $\KK$, we infer the existence of the infinitely many closed geodesics asserted by Theorem~\ref{mt:multiplicity}. The argument for Theorem~\ref{mt:sphere} similarly employs a convenient Riemannian metric to establish the non-vanishing of the local homology of the flat knot type consisting of simple loops intersecting each of the two $\gamma_i$'s in two points.

\subsection{Existence of Birkhoff sections}

While Theorems~\ref{mt:multiplicity} and~\ref{mt:sphere} have independent interest, our main motivation was to combine them with a recent work of the second author together with Contreras, Knieper, and Schulz \cite{Contreras:2022ab} in order to establish a full, unconditional, existence result for Birkhoff sections of geodesic flows of closed oriented surfaces. In order to state the result, let us recall the relevant definitions and the state of the art around this problem.

Let $\psi_t:N\to N$ be the flow of a nowhere vanishing vector field $X$ on a closed 3-manifold $N$. A surface of section is a (possibly disconnected) immersed compact surface $\Sigma \looparrowright N$ whose boundary $\partial \Sigma$ consist of periodic orbits of $\psi_t$, while the interior $\interior(\Sigma)$ is embedded in $N \setminus \partial \Sigma$ and transverse to $X$. Such a $\Sigma$ is called a \emph{Birkhoff section} when there exists $T>0$ such that, for each $z\in N$, the orbit segment $\psi_{[0,T]}(z)$ intersects $\Sigma$. By means of a Birkhoff section, the study of the dynamics of $\psi_t$, apart from the finitely many periodic orbits in $\partial\Sigma$, can be reduced to the study of the surface diffeomorphism $\interior(\Sigma)\to\interior(\Sigma)$, $z\mapsto\psi_{\tau(z)}(z)$, where 
\[\tau(z)=\min\big\{t\in(0,T]\ \big|\ \psi_t(z)\in\Sigma \big\}.\]
This reduction is highly desirable, as there are powerful tools that allow to study the dynamics of diffeomorphisms specifically in dimension two (e.g.~Poincaré-Birkhoff fixed point theorem, Brouwer translation theorem, Le Calvez transverse foliation theory, etc.).

The notion of Birkhoff section was first introduced by Poincaré in his study of the circular planar restricted three-body problem, but owes its name to the seminal work of Birkhoff \cite{Birkhoff:1917aa}, who established their existence for all geodesic flows of closed orientable surfaces with nowhere vanishing curvature. Over half a century later, a result of Fried \cite{Fried:1983aa} confirmed the existence of Birkhoff sections for all transitive Anosov flows of closed 3-manifolds. In one of the most famous articles from symplectic dynamics \cite{Hofer:1998aa}, Hofer, Wysocky, and Zehnder proved that the Reeb flow of any 3-dimensional convex contact sphere admits a Birkhoff section that is an embedded disk. Since this work, the quest for Birkhoff sections of Reeb flows has been a central theme in symplectic dynamics, see e.g.~\cite{Salomao:2018aa, Hryniewicz:2023aa}. Recently, the existence of Birkhoff sections for the Reeb vector field of a $C^\infty$-generic contact form of any closed 3-manifold has been confirmed independently by the first author and Contreras \cite{Contreras:2022aa}, and by Colin, Dehornoy, Hryniewicz, and Rechtman \cite{Colin:2024aa}. Indeed, the existence of Birkhoff sections was proved for non-degenerate contact forms satisfying  any of the following assumptions, which hold for a $C^\infty$-generic contact form: the transversality of the stable and unstable manifolds of the hyperbolic closed orbits \cite{Contreras:2022aa}, or the equidistribution of the closed orbits \cite{Colin:2024aa}. These results extended in different directions  previous work of Colin, Dehornoy, and Rechtman \cite{Colin:2023aa}, which in particular provided, for any non-degenerate Reeb flows of any closed 3 manifold, a surface of section $\Sigma$ that is almost a Birkhoff section, except for some escaping half-orbits converging to hyperbolic boundary components of $\Sigma$. This result, in turn, relies on Hutchings' embedded contact homology \cite{Hutchings:2014vp}, a powerful machinery based on holomorphic curves and Seiberg-Witten theory, which provides plenty of surfaces of section almost filling the whole ambient 3-manifold. Beyond the above generic conditions, the existence of a Birkhoff section for the Reeb flow of \emph{any} contact form on any closed 3-manifold remains an open problem.

For the special case of geodesic flows of closed Riemannian surfaces, the non-degeneracy and the transversality of the stable and unstable manifolds of the hyperbolic closed geodesics hold for a $C^\infty$ generic Riemannian metric, and so does the existence of Birkhoff sections according to the above mentioned result in \cite{Contreras:2022aa}. In a recent work of the first author together with Contreras, Knieper, and Schulz \cite{Contreras:2022ab}, this latter result was re-obtained without holomorphic curves techniques, employing instead the curve shortening flow. Actually, the existence result obtained is slightly stronger: the non-degeneracy is only required for the contractible simple closed geodesics without conjugate points. Our third main result removes completely any generic requirement.

\begin{MainThm}\label{mt:Birkhoff_sections}
The geodesic flow of any closed orientable Riemannian surface admits a Birkhoff section.
\end{MainThm}

The scheme of the proof is the following. Any closed geodesic produces two immersed surfaces of section of annulus type, the so-called Birkhoff annuli, consisting of all unit tangent vectors based at any point of the closed geodesic and pointing on one of the two sides of it. A surgery procedure due to Fried \cite{Fried:1983aa} allows to glue together all Birkhoff annuli of a suitable collection of non-contractible simple closed geodesics, producing a surface of section $\Sigma$. A contractible simple closed geodesic without conjugate points $\gamma$ whose unit-speed lifts $\pm\dot\gamma$ do not intersect $\Sigma$ is an obstruction for $\Sigma$ to be a Birkhoff section. Even after adding to $\Sigma$ the Birkhoff annuli $A^+\cup A^-$ of $\gamma$, there may still be half-orbits of the geodesic flow converging to $\pm\dot\gamma$ without intersecting $\Sigma':=\Sigma\cup A^+\cup A^-$. Theorems~\ref{mt:multiplicity} and~\ref{mt:sphere} allow us to always detect other closed geodesics intersecting $\gamma$ transversely, and after gluing their Birkhoff annuli to $\Sigma'$ we obtain a new surface of section $\Sigma''$ that does not have $\pm\dot\gamma$ as $\omega$-limit of half-orbits not intersecting $\Sigma''$. As it turns out, after repeating this procedure for finitely many contractible simple closed geodesics without conjugate points, we end up with a Birkhoff section.

\begin{Remark}
The argument in our proof does not allow us to control the genus of the Birkhoff section $\Sigma\looparrowright SM$ provided by Theorem~\ref{mt:Birkhoff_sections} for the geodesic flow of a closed connected orientable Riemannian surface $(M,g)$. Nevertheless, it allows us to bound from above the number of connected components $b$ of the boundary $\partial\Sigma$ as
\[ b\leq 8\max\big\{1,\mathrm{genus}(M)\big\}+\frac{4}{\pi}\,\area(M,g)\max(R_g), \]
where $R_g$ denotes the Gaussian curvature, see Remark~\ref{r:bound_boundary_components}.
It is worthwhile to mention a recent work of Kim, Kim, and van Koert \cite{Kim:2022aa},  which exhibits examples of Reeb flows on homology 3-spheres such that the minimal number of boundary components of their Birkhoff sections is arbitrarily large.
\end{Remark}

\subsection{Organization of the paper}
In Section~\ref{s:preliminaries} we recall the needed background on the curve shortening flow, and on the classical variational setting for the closed geodesics problem. 
In Section~\ref{s:Morse}, we develop the analogue of the classical results from Morse theory of closed geodesics within a primitive relative flat knot type.
In Section~\ref{s:flat_links}, we first prove the simpler, special case of Theorem~\ref{mt:links} for primitive flat knot types, actually under slightly weaker assumptions (Theorem~\ref{t:knots}). Next, after suitable preliminaries, we prove a slightly stronger version of Theorem~\ref{mt:links},  replacing the non-degeneracy of the original flat link of closed geodesics with a homological visibility assumption  (Theorem~\ref{t:links}). In Section~\ref{s:forcing} we prove Theorem~\ref{mt:multiplicity} and~\ref{mt:sphere}, and in the final Section~\ref{s:Birkhoff_sections} we prove Theorem~\ref{mt:Birkhoff_sections}.

\subsection{Acknowledgments}
The first author thanks Matthias Meiwes and Abror
Pirnapasov for  many illuminating and inspiring conversations about forcing in dynamics, and Yuri Lima for a conversation about the first return time of Birkhoff sections.

\section{Preliminaries}\label{s:preliminaries}

\subsection{Curve shortening flow}
\label{ss:curve_shortening}

Let $(M,g)$ be a closed Riemannian surface.
The curve shortening flow $\phi^t=\phi^t_g$, for $t\geq0$, is a continuous semi-flow on $\Imm(S^1,M)$ defined as follows: its orbits $\gamma_t:=\phi^t(\gamma_0)$ are solutions of the PDE
\[
\partial_t\gamma_t=\kappa_{\gamma_t}n_{\gamma_t},
\]
for all $t\in[0,t_\gamma)$. Here, $t_\gamma=t_{g,\gamma}$ is the  extended real number giving the maximal interval of definition, $n_{\gamma_t}:[0,1]\to TM$ is any vector field along $\gamma_t$ that is orthonormal to $\dot\gamma_t$, and $\kappa_{\gamma_t}:S^1\to\R$ denotes the signed geodesic curvature of $\gamma_t$ with respect to $n_{\gamma_t}$, i.e.
\[
\nabla\!_s\tfrac{\dot\gamma_t(s)}{\|\dot\gamma_t(s)\|_g}
=
\kappa_{\gamma_t}(s)\,\|\dot\gamma_t(s)\|_g\,n_{\gamma_t}(s),\]
where $\nabla_s$ denotes the Levi-Civita covariant derivative. Notice that we may have $n_{\gamma_t}(0)=-n_{\gamma_t}(1)$ if $M$ is not orientable, but nevertheless the product $\kappa_{\gamma_t} n_{\gamma_t}$ is independent of the choice of $n_{\gamma_t}$.
We recall that $\Imm(S^1,M)$ is endowed with the $C^3$ topology, as we specified in the introduction.
The map $(t,\gamma)\mapsto\phi^t(\gamma)$ is continuous on its domain of definition, which is an open neighborhood of $\{0\}\times\Imm(S^1,M)$ in $[0,\infty)\times\Imm(S^1,M)$.
The curve shortening flow is $\Diff(S^1)$-equivariant, meaning that 
\begin{align*}
\phi^t(\gamma\circ\theta)
=
\phi^t(\gamma)\circ\theta,
\qquad\forall\gamma\in\Imm(S^1,M),\ \theta\in\Diff(S^1),\ t\in[0,t_\gamma).
\end{align*}
In particular, it also induces a continuous semi-flow on the quotient 
\[\Omega=\frac{\Imm(S^1,M)}{\Diff_+(S^1)}\] that we still denote by $\phi^t$ (we will mainly consider $\phi^t$ defined on $\Omega$, except if we work with an explicit parametrization of the initial loop $\gamma$). 

\begin{Remark}
The fact that the loops in $\Omega$ are oriented is not particularly relevant for us, and in most sections of the article we could have equivalently worked with the space of unoriented loops $\Imm(S^1,M)/\Diff(S^1)$. The only reason to consider the space of oriented loops $\Omega$ is that it makes the local homology computation \eqref{e:loc_hom_circle} of Lemma~\ref{l:non_deg_visible} simpler.
\end{Remark}

The curve shortening flow is the $L^2$ anti-gradient flow of the length functional $L=L_g:\Omega\to(0,\infty)$. More specifically, it satisfies
\begin{align}
\label{e:derivative_length}
\frac{d}{dt} L(\phi^t(\gamma))
=
-\int_{S^1} \kappa_\gamma(s)^2 \|\dot\gamma(s)\|_g\,ds
\leq0,
\end{align}
and the equality holds if and only if $\gamma$ is a closed geodesic, which is the case if and only if $\gamma=\phi^t(\gamma)$ for all $t\geq0$.

According to a theorem of Grayson \cite{Grayson:1989aa}, if the orbit $\phi^t(\gamma)$ of an embedded loop is only defined on a bounded interval $[0,t_\gamma)\subsetneq[0,\infty)$, then $\phi^t(\gamma)$ shrinks to a point as $t\to t_\gamma$; if instead $[0,t_\gamma)=[0,\infty)$, then $L(\phi^t(\gamma))\to\ell>0$,  the curvature $\kappa_{\phi^t(\gamma)}$ converges to zero in the $C^\infty$ topology as $t\to\infty$, and in particular there exists a subsequence $t_n\to\infty$ such that $\phi^{t_n}(\gamma)$ converges to a closed geodesic. This is not necessarily the case if $\gamma$ is not embedded, as it may also happen that $L(\phi^t(\gamma))\to\ell>0$ and $\phi^t(\gamma)$ develops a singularity as $t\to t_\gamma$. Nevertheless, the forthcoming lemma due to Angenent allows to control this behavior.

Let $\gamma:S^1\looparrowright M$ be an immersed loop such that the restriction $\gamma|_{[s_1,s_2]}$ is an embedded subloop, i.e.~$\gamma(s_1)=\gamma(s_2)$ and $\gamma|_{[s_1,s_2)}$ is injective. We say that $\gamma|_{[s_1,s_2]}$ is a \emph{$\rho$-subloop} when it bounds a disk $D\subset M$ of area less than or equal to $\rho$, and for some $\delta>0$ the curves $\gamma|_{(s_1-\delta,s_1)}$ and $\gamma|_{(s_2,s_2+\delta)}$ do not enter $D$ (see Figure~\ref{f:subloop}). With the same notation of the introduction, we denote by $\Delta=\Delta_1$ the closed subspace of $\Omega$ consisting of those loops having a self-tangency.

\begin{figure}
\begin{footnotesize}
\includegraphics{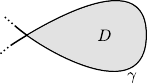}
\end{footnotesize}
\caption{A $\rho$-subloop of $\gamma$, with a filling  $D$ of area less than or equal to $\rho$.}
\label{f:subloop}
\end{figure}

\begin{LemmaB}[\cite{Angenent:2005aa}, Lemmas 5.3-4]\label{l:subloops}$ $
\begin{itemize}\setlength{\itemsep}{5pt}
 \item[(i)] For each $\gamma\in\Omega$ such that $L(\phi^t(\gamma))\to\ell>0$ as $t\to t_{\gamma}$, either $\phi^{t_n}(\gamma)$ converges to a closed geodesic for some sequence $t_n\to t_\gamma$, or $\phi^t(\gamma)$ develops a singularity as $t\to t_\gamma$. In this latter case, for each  $\rho>0$ and for each $t$ sufficiently close to $t_\gamma$ the loop $\phi^t(\gamma)$ possesses a $\rho$-subloop.

 \item[(ii)] There exists $\rho=\rho_g>0$ with the following property. Let $\gamma\in\Omega$ and $\tau>0$ be such that $\phi^t(\gamma)\in\Omega\setminus\Delta$ for all $t\in[0,\tau]$, and assume that $\gamma|_{[a_0,b_0]}$ is a $\rho$-subloop $($for some parametrization on $\gamma$$)$. Then, $[a_0,b_0]$ can be extended to a continuous family of intervals $[a_t,b_t]$ such that $\phi^t(\gamma)|_{[a_t,b_t]}$ is a $(\rho-\tfrac\pi2 t)$-subloop for all $t\in[0,\tau]$.
 \hfill\qed
\end{itemize}
\end{LemmaB}

\subsection{Primitive flat knot types} \label{ss:primitive}

Let $\bm\zeta=(\zeta_1,...,\zeta_n)$ be either the empty link (when $n=0$), or a flat link of pairwise geometrically distinct closed geodesics (namely, $\zeta_i$ and $\zeta_j$ are transverse for all $i\neq j$). We denote by $\Delta(\bm\zeta)$ the closed subset of $\Omega$ consisting of those loops $\gamma$ having a self-tangency or a tangency with some component of $\bm\zeta$ (Figure~\ref{f:tangency}). With the notation of Section~\ref{ss:C0_stability}, we have
\begin{align*}
 \Delta(\bm\zeta) 
 = 
 \big\{
 \gamma\in\Omega
 \ \big|\ 
 (\gamma,\bm\zeta)\in\Delta_{n+1}
 \big\}.
\end{align*}
A path-connected component $\KK$ of $\Omega\setminus\Delta(\bm\zeta)$ is called a \emph{flat knot type relative} to $\bm\zeta$ (notice that, if $\bm\zeta$ is empty, this notion reduces to the one of ordinary flat knot type). We denote by $\overline{\KK}$ and $\partial\KK$ its closure and its boundary in $\Omega$ respectively. The relative flat knot type $\KK$ is called \emph{primitive} when $\partial\KK$ does not contain non-primitive loops nor components of $\bm\zeta$.

Throughout this section, we fix a primitive flat knot type $\KK$ relative to $\bm\zeta$. We recall the main properties of the curve shortening flow with respect to $\KK$, established by Angenent.

\begin{LemmaB}[\cite{Angenent:2005aa}, Lemmas 3.3 and 6.2]\label{l:Angenent_primitive}
$ $
\begin{itemize}
\setlength{\itemsep}{5pt}

\item[(i)] If $\gamma\in\KK$ and $\phi^{t_0}(\gamma)\not\in\overline\KK$ for some $t_0\in(0,t_\gamma)$, then $\phi^t(\gamma)\not\in\overline\KK$ for all $t\in(t_0,t_\gamma)$ as well.

\item[(ii)] If $\gamma,\phi^{t_0}(\gamma)\in\overline\KK$ for some $t_0\in(0,t_\gamma)$, then $\phi^t(\gamma)\in\KK$ for all $t\in(0,t_0)$.\hfill\qed

\end{itemize}
\end{LemmaB}

\begin{Remark}
It can be easily shown that the curve shortening flow preserves both the subspace of $m$-th iterates $\Omega^m:=\{\gamma^m\ |\ \gamma\in\Omega\}$ and its complement $\Omega\setminus\Omega^m$. This shows that the assumption that $\KK$ is primitive is essential at least for point (ii) of Lemma~\ref{l:Angenent_primitive}.
\end{Remark}

\begin{figure}
\begin{footnotesize}
\includegraphics{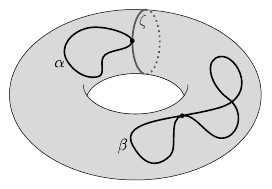}
\end{footnotesize}
\caption{A loop $\alpha\in\Delta(\zeta)$ with a tangency to $\zeta$, and a loop $\beta\in\Delta(\zeta)$ with a self-tangency. We stress that, while the tangencies depicted are without crossing, tangencies can be topologically transverse.}
\label{f:tangency}
\end{figure}

We define the \textbf{exit set} of the primitive flat knot type $\KK$ as
\begin{align*}
 \partial_-\KK
 &:=
 \big\{
 \gamma\in\partial\KK\ \big|\ \phi^t(\gamma)\not\in\overline\KK\mbox{ for all }t\in(0,t_\gamma)
 \big\}.
\end{align*}
By Lemma~\ref{l:Angenent_primitive}, $\partial_-\KK$ is a closed subset of $\partial\KK$.

\begin{Lemma}[\cite{Angenent:2005aa}, Lemma~6.3]
The exit set $\partial_-\KK$ does not depend on the Riemannian metric $g$.
\hfill\qed
\end{Lemma}

We denote by $\Omega_\rho=\Omega_{g,\rho}$ the open subset of $\Omega$ consisting of those $\gamma$ containing a $\rho$-subloop, and we set
\begin{align}
\label{e:exit_and_subloops}
 \overline\KK_{\rho}=\overline\KK_{g,\rho}:=\partial_-\KK\cup\big(\overline\KK\cap\Omega_{\rho}\big).
\end{align}

\begin{Lemma}[\cite{Angenent:2005aa}, Prop.~6.8]\label{l:Angenent_exit_time}
The exit-time function
\begin{align*}
\tau_{\rho}=\tau_{g,\rho}:\overline\KK\to[0,\infty],
\qquad
\tau_\rho(\gamma)=\inf\big\{t\in[0,t_\gamma)\ \big|\ \phi^t(\gamma)\in\overline\KK_{\rho}\big\}
\end{align*}
is continuous $($here, we employ the usual convention $\inf\varnothing=\infty$$)$.\hfill\qed
\end{Lemma}

The continuity of the exit-time function, together with Lemma~\ref{l:subloops}(ii), implies that the inclusion $(\overline\KK,\overline\KK_{\rho_1})\hookrightarrow(\overline\KK,\overline\KK_{\rho_2})$ is a homotopy equivalence for all $\rho_1<\rho_2<\rho_g$. Indeed, its homotopy inverse is given by $\gamma \mapsto \phi^{\tau_{\rho_1,\rho_2}(\gamma)}(\gamma)$,
where
\begin{align*}
 \tau_{\rho_1,\rho_2}(\gamma)
 :=
 \min\big\{\tau_{\rho_1}(\gamma),\tfrac2\pi(\rho_2-\rho_1)\big\}.
\end{align*}

\subsection{Curvature control}\label{ss:curvature}

In the already mentioned work \cite{Grayson:1989aa}, Grayson described the behavior of the curvature of embedded loops that evolve under the curve shortening flow. Actually, the analysis does not require the embeddedness and holds for immersed loops as well, and even in the more general setting of reversible Finsler metrics \cite[Section~2.5]{De-Philippis:2022aa}. We summarize here the results needed later on, in Section~\ref{ss:nbhds}, for constructing the suitable neighborhoods of compact sets of closed geodesics that enter the definition of local homology, one of the ingredients of our proof of Theorems~\ref{mt:links} and~\ref{mt:multiplicity}.

In order to simplify the notation, for any given immersed loop $\gamma_0:S^1\looparrowright M$, we denote by $\gamma_t:=\phi^t(\gamma_0)$ its evolution under the curve shortening flow, and by $\kappa_t:=\kappa_{\gamma_t}$ the signed geodesic curvature of $\gamma_t$ with respect to a normal vector field. In \cite[Lemma~1.2]{Gage:1990aa}, Gage showed that $\kappa_t$ evolves according to the PDE
\begin{align}
\label{e:PDE_curvature}
\partial_t\kappa_t=D^2\kappa_t+\kappa_t r_t+\kappa_t^3, 
\end{align}
where $r_t:S^1\to\R$ denotes the Gaussian curvature of the Riemannian surface $(M,g)$ along $\gamma_t$, and $D=\|\dot\gamma_t(s)\|_g^{-1}\partial_s$ is a vector field on $[0,t_{\gamma_0})\times S^1$, which we see as a differential operator acting on functions $f:[0,t_{\gamma_0})\times S^1\to\R$, $f(t,s)=f_t(s)$ as 
\[Df_t(s)=\|\dot\gamma_t(s)\|_g^{-1}\dot f_t(s).\]
We denote by $\Gamma_t:=\gamma_t\circ\nu_t^{-1}:[0,L(\gamma_s)]\looparrowright M$ the arclength reparametrization of $\gamma_t$, where $\nu_t:[0,1]\to[0,L(\gamma_t)]$ is the function 
\begin{align*}
 \nu_t(s) = \int_0^s \|\dot\gamma_t(r)\|_g\,dr,
\end{align*}
and by $K_t:=\kappa_t\circ\nu_t^{-1}$ the signed geodesic curvature of $\Gamma_t$.

\begin{Prop}\label{p:curvature_control}
For each compact interval $[a,b]\subset(0,\infty)$ there exist $c>0$ such that, for each $\epsilon>0$ small enough, for each immersed smooth loop $\gamma_0:S^1\looparrowright M$ satisfying $L(\gamma_0)\in[a,b]$ and $\|K_0\|_{L^2}+\|\dot K_0\|_{L^2}\leq\epsilon$, and for each $t\in[0,t_\gamma)$, we have $\|K_t\|_{L^\infty}\leq c\,\epsilon$ or $L(\gamma_0)-L(\gamma_t)\geq \epsilon^2$.
\end{Prop}

\begin{proof}
We will fix an upper bound for the quantity $\epsilon>0$ later on. For now, we consider $\epsilon\in(0,\sqrt{a/2})$ together with the data stated in the lemma and the notation introduced just before. Notice that 
$\dot K_t\circ\nu_t = D\kappa_t$ and $\ddot K_t\circ \nu_t=D^2\kappa_t$.
We denote by $R_t:=r_t\circ\nu_t^{-1}$ the Gaussian curvature of $(M,g)$ along $\Gamma_t$. We employ~\eqref{e:PDE_curvature} to compute
\begin{align*}
 \partial_t \|K_t\|_{L^2}^2
 &=
 \partial_t \int_{0}^{L(\gamma_t)}K_t^2 ds
 =
 \partial_t \int_{S^1}\kappa_t^2\|\dot\gamma_t\|_g ds\\
 &=
 \int_{S^1}\Big(2\kappa_t\,\partial_t\kappa_t\|\dot\gamma_t\|_g + \kappa_t^2\partial_t\|\dot\gamma_t\|_g\Big)\,ds\\
 &=
 \int_{S^1}\Big(2\kappa_t(D^2\kappa_t+\kappa_tr_t+\kappa_t^3)- \kappa_t^4\Big)\|\dot\gamma_t\|_g\,ds\\
 &\leq
 -2 \|\dot K_t\|_{L^2}^2 + 2 \|K_t^2R_t\|_{L^1} + \|K_t^4\|_{L^1}\\
 & \leq
 -2 \|\dot K_t\|_{L^2}^2 +  \|K_t\|_{L^2}^2\Big( 2 \|R_t\|_{L^\infty} + \|K_t\|_{L^\infty}^2 \Big).
\end{align*}
If $L(\gamma_t)>a/2$, we can bound from above the term $\|K_t\|_{L^\infty}^2$ by
\begin{align*}
 \|K_t\|_{L^\infty}^2 
  \leq  \frac2{L(\gamma_t)}\|K_t\|_{L^2}^2 + 2L(\gamma_t)\|\dot K_t\|_{L^2}^2
 \leq  2a^{-1}\|K_t\|_{L^2}^2 + 2b\|\dot K_t\|_{L^2}^2 .
\end{align*}
We use this inequality as a lower bound for $\|\dot K_t\|_{L^2}^2$, and continuous the previous estimate as
\begin{equation}\label{e:estimate_Kt_L2_Zoll}
\begin{split}
\!\!\!\! \partial_t \|K_t\|_{L^2}^2
 &\leq
 \big( 4a^{-1}b^{-1} + 2\|R_t\|_{L^\infty}\big) \|K_t\|_{L^2}^2 + \|K_t\|_{L^\infty}^2 \bigg( \|K_t\|_{L^2}^2 - b^{-1}\bigg) \\
 &\leq 
 c_1 \|K_t\|_{L^2}^2 + \|K_t\|_{L^\infty}^2 \big( \|K_t\|_{L^2}^2 - c_1^{-1}\big),
\end{split}
\end{equation}
where $c_1\geq1$ is a constant depending only on the compact interval $[a,b]$ and  on the Riemannian metric $g$. 

We further require $\epsilon^2< c_1^{-1}e^{-c_1}$, and set
\begin{align*}
 \tau:=\sup\big\{ t\in[0,t_{\gamma_0})\ \big|\ \|K_t\|_{L^2}^2\leq c_1^{-1},\ L(\gamma_0)-L(\gamma_t)\leq\epsilon^2 \big\}.
\end{align*}
The inequality~\eqref{e:estimate_Kt_L2_Zoll} implies 
\begin{align}\label{e:estimate_Kt_L2_Zoll_final}
 \partial_t \|K_t\|_{L^2}^2
 < 
 c_1 \|K_t\|_{L^2}^2,
 \qquad \forall t\in[0,\tau).
\end{align}
Notice that \eqref{e:derivative_length} can be rewritten in terms of $K_t$ as
$\tfrac{d}{dt} L(\gamma_t) = -\|K_t\|_{L^2}^2$. This, together with~\eqref{e:estimate_Kt_L2_Zoll_final} and with the initial bound $\|K_{0}\|_{L^2}\leq\epsilon$, gives
\begin{align*}
 \|K_{t}\|_{L^2}^2 
 \leq 
 \|K_{0}\|_{L^2}^2 
 +c_1\underbrace{\int_0^t \|K_{r}\|_{L^2}^2\,dr}_{=L(\gamma_0)-L(\gamma_t)}
 \leq
 \underbrace{(c_1+1)}_{=:c_2}\epsilon^2,\qquad
 \forall t\in[0,\tau).
\end{align*}

We claim that, if $\tau<t_{\gamma_0}$,  
\begin{align*}
 L(\gamma_0)-L(\gamma_\tau)=\epsilon^2,
\end{align*}
so that $L(\gamma_0)-L(\gamma_t)>\epsilon^2$ for all $t\in(\tau,t_{\gamma_0})$.
Assume by contradiction that this does not hold, so that $\|K_\tau\|_{L^2}^2= c_1^{-1}$.
By \eqref{e:estimate_Kt_L2_Zoll_final} and Gronwall inequality, we have 
\[\|K_{\tau}\|_{L^2}^2 \leq e^{c_1(\tau-t)} \|K_{t}\|_{L^2}^2,\qquad\forall t\in[0,\tau].\] 
Therefore
$c_1^{-1}=\|K_{\tau}\|_{L^2}^2\leq e^{c_1\tau}\|K_0\|_{L^2}^2\leq e^{c_1\tau}\epsilon^2\leq c_1^{-1}e^{c_1(\tau-1)}$, 
and we infer that $\tau>1$.
This further implies
\begin{align*}
 c_1^{-1} = \|K_{\tau}\|_{L^2}^2 \leq e^{c_1(\tau-t)} \|K_{t}\|_{L^2}^2 \leq e^{c_1} \|K_{t}\|_{L^2}^2,\qquad\forall t\in[\tau-1,\tau].
\end{align*}
By integrating with respect to $t$ on $[\tau-1,\tau]$, we obtain
\begin{align*}
 e^{-c_1}c_1^{-1}
 \leq
 \int_{\tau-1}^{\tau} \|K_{t}\|_{L^2}^2dt
 =
 L(\gamma_{\tau-1})- L(\gamma_{\tau})
 \leq 
 \epsilon^2,
\end{align*}
which contradicts the fact that $\epsilon^2<e^{-c_1}c_1^{-1}$.

Next, we compute
\begin{align*}
 \partial_t D\kappa_t
 &=
 \frac{\partial_t \dot \kappa_t}{\|\dot\gamma_t\|_g}
 -
 \frac{\dot \kappa_t\,\partial_t\|\dot\gamma_t\|_g}{\|\dot\gamma_t\|_g^2}
 =
 D\partial_t \kappa_t + \frac{\dot \kappa_t\,\kappa_t^2}{\|\dot\gamma_t\|_g}\\
 &=
 D(D^2 \kappa_t + \kappa_t r_t + \kappa_t^3) + D\kappa_t\,\kappa_t^2\\
 &=
 D^3 \kappa_t + D\kappa_t\, r_t + \kappa_t\,Dr_t + 4\kappa_t^2\,D\kappa_t .
\end{align*}
and obtain the bound
\begin{align*}
 \partial_t \|\dot K_t\|_{L^2}^2
 &=
 \partial_t \int_{S^1} (D\kappa_t)^2\|\dot\gamma_t\|_g\,ds
 =
 \int_{S^1} \big( 2 D\kappa_t\,\partial_tD\kappa_t - (D\kappa_t)^2\kappa_t^2 \big)\|\dot\gamma_t\|_g\,ds\\
 &=
 \int_{S^1} D\kappa_t\big( 2 D^3\kappa_t + 2D\kappa_t\,r_t + 2\kappa_t\,Dr_t + 7D\kappa_t\,\kappa_t^2 \big)\|\dot\gamma_t\|_g\,ds\\
 &=
 \int_0^{L(\gamma_t)} \Big(\! -2 \ddot K_t^2 + 2\dot K_t^2 R_t + 2K_t \dot K_t \dot R_t + 7\dot K_t^2 K_t^2 \Big)\,ds\\
 &=
 \int_0^{L(\gamma_t)} \Big(\! -2 \ddot K_t^2 - 2 K_t \ddot K_t R_t + 7\dot K_t^2 K_t^2 \Big)\,ds\\ 
 &\leq -2 \|\ddot K_t\|_{L^2}^2 + 2 \|R_t\|_{L^\infty} \|K_t \ddot K_t\|_{L^1}    + 7 \|\dot K_t K_t\|_{L^2}^2.
\end{align*}
We denote by $\delta\in(0,1)$ a small constant that we will fix later, and set
\begin{align*}
 I:=\big\{ t\in[0,\tau)\ \big|\  \|K_t\|_{L^2} \leq \delta \|\dot K_t\|_{L^2} \big\}.
\end{align*}
For each $t\in I$, since
\begin{align*}
\|\dot K_t\|_{L^2}^2
\leq 
\|K_t\|_{L^2}\|\ddot K_t\|_{L^2}
\leq
\delta\|\dot K_t\|_{L^2}\|\ddot K_t\|_{L^2},
\end{align*}
we have $\|\dot K_t\|_{L^2}\leq \delta\|\ddot K_t\|_{L^2}$. Therefore
\begin{equation}
\label{e:Cinfty_Zoll_estimates_2_1}
\begin{split}
 \|\dot K_t K_t\|_{L^2}
 &\leq
 \|K_t\|_{L^2} \|\dot K_t\|_{L^\infty}\\
 &\leq
 \|K_t\|_{L^2} \big(\|\dot K_t\|_{L^2} + L(\gamma_t)\|\ddot K_t\|_{L^2}\big)\\
 &\leq
 c_2\,\epsilon\,(\delta+b)\|\ddot K_t\|_{L^2}, 
\end{split}
\end{equation}
and, using Peter-Paul inequality,\index{Peter-Paul inequality}
\begin{equation}
\label{e:Cinfty_Zoll_estimates_2_2}
\begin{split}
 \|K_t \ddot K_t\|_{L^1}
 &\leq
 \tfrac12\big(\delta^{-1}\|K_t\|_{L^2}^2+ \delta\|\ddot K_t\|_{L^2}^2\big)
 \leq
 \delta^3 \|\ddot K_t\|_{L^2}^2.
\end{split}
\end{equation}
Now, we require $\epsilon$ and $\delta$ to be small enough (depending only on the compact interval $[a,b]$ and  on the Riemannian metric $g$) so that, plugging~\eqref{e:Cinfty_Zoll_estimates_2_1} and~\eqref{e:Cinfty_Zoll_estimates_2_2} into the above estimate of $\partial_t\|\dot K_t\|_{L^2}^2$, we obtain
\begin{align}
\label{e:final_Grayson_Zoll}
 \partial_t\|\dot K_t\|_{L^2}^2
 \leq
 -\|\ddot K_t\|_{L^2}^2
 \leq
 0,
 \qquad\forall t\in I.
\end{align}

Since $\|K_0\|_{L^2}+\|\dot K_0\|_{L^2}\leq\epsilon$, we have
\begin{align}
\label{e:bound_dot_K_t_outside_I}
\|\dot K_t\|_{L^2}^2\leq \delta^{-2}\|K_t\|_{L^2}^2\leq \underbrace{\delta^{-2}c_2}_{=:c_3}\epsilon^2,
\qquad\forall t\in\overline{[0,\tau)\setminus I}.
\end{align}
For each $t\in I$, if $r\in[0,\tau)$ is the minimal value such that $[r,t]\in I$, the inequalities~\eqref{e:final_Grayson_Zoll} and~\eqref{e:bound_dot_K_t_outside_I} imply
\begin{align*}
 \|\dot K_t\|_{L^2}^2
 \leq
 \|\dot K_r\|_{L^2}^2
 \leq
 \delta^{-2}\|K_{r}\|_{L^2}^2
 \leq
 \underbrace{\delta^{-2} c_3}_{=:c_4} \epsilon^2.
\end{align*}
Overall, we obtained the inequality $\|\dot K_t\|_{L^2}^2\leq c_4\epsilon^2$ for all $t\in[0,\tau)$, and we conclude
\begin{align*}
 \|K_t\|_{L^\infty}
 &\leq
 \|\dot K_t\|_{L^1} + L(\gamma_t)^{-1}\|K_t\|_{L^1}\\
 &\leq
 L(\gamma_t)^{1/2}\|\dot K_t\|_{L^2} + L(\gamma_t)^{-1/2}\|K_t\|_{L^2}\\
 &\leq\big(b\,c_4 + a^{-1}c_2\big)^{1/2}\epsilon.
 \qedhere
\end{align*}
\end{proof}

\subsection{The energy functional}\label{ss:energy}

The conventional setting for the variational theory of the closed geodesics in a closed Riemannian manifold, and in particular in our closed Riemannian surface $(M,g)$, is the one of the free loop space \[\Lambda:=W^{1,2}(S^1,M).\] The energy functional $E:\Lambda\to[0,\infty)$, which is defined as
\begin{align*}
 E(\gamma)=\int_{S^1} \|\dot\gamma(t)\|_g^2\,dt,
\end{align*}
is smooth, and its critical points are the 1-periodic solutions of the ODE $\nabla_t\dot\gamma\equiv0$, where $\nabla_t$ is the Levi-Civita covariant derivative. Namely, the critical points are the constant curves and the 1-periodic geodesics (parametrized with constant speed). 
The space of constant curves is the level set $E^{-1}(0)$ of global minima, and we denote by \[\crit^+(E)=\crit(E)\cap E^{-1}(0,\infty)\] the space of closed geodesics, which are the non-trivial critical points of $E$. 
At these critical points, the energy functional is related to the length functional by
\begin{align*}
 E(\gamma)^{1/2}=L(\gamma),\qquad\forall \gamma\in\crit^+(E).
\end{align*}
The circle $S^1$ acts on $\Lambda$ by time translation, i.e.~$t\cdot\gamma=\gamma(t+\cdot)$ for all $t\in S^1$ and $\gamma\in\Lambda$, and the energy $E$ is $S^1$-invariant. In particular, every $\gamma\in\crit^+(E)$ belongs to a circle of critical points $S^1\cdot\gamma\subset\crit^+(E)$. A closed geodesic $\gamma\in\crit^+(E)$ is called \emph{isolated} when $S^1\cdot\gamma$ is an isolated circle in $\crit^+(E)$.

It is often convenient to consider finite dimensional approximations of the free loop space $\Lambda$, given by the spaces
\begin{align*}
\Lambda_k:=\Big\{ \xx=(x_0,...,x_{k-1})\in  M^{\times k}\ \Big|\ d(x_i,x_{i+1})<\inj(M,g)\ \ \forall i\in\Z_k \Big\},
\end{align*}
where $k\geq 3$ is an integer, $d$ denotes the Riemannian distance, and $\inj(M,g)$ the injectivity radius. We see $\Lambda_k$ as a subspace of $\Lambda$ consisting of broken geodesic loops: namely, each $\xx$ corresponds to the piecewise smooth loop $\gamma_{\xx}\in\Lambda$ such that each restriction $\gamma_{\xx}|_{[i/k,(i+1)/k]}$ is the shortest geodesic segment joining $\gamma_{\xx}(i/k)=x_i$ and $\gamma_{\xx}((i+1)/k)=x_{i+1}$. Analogously, any tangent vector $\vv=(v_0,...,v_{k-1})\in T_{\xx}\Lambda_k$ corresponds to a unique piecewise smooth 1-periodic vector field $J_{\vv}\in T_{\gamma_{\xx}}\Lambda$  such that each restriction $J_{\vv}|_{[i/k,(i+1)/k]}$ is the unique Jacobi fields with endpoints $J_{\vv}(i/k)=v_i$ and $J_{\vv}((i+1)/k)=v_{i+1}$.

Consider a closed geodesic $\gamma\in\crit^+(E)$, and fix the integer $k$ to be large enough so that $E(\gamma)^{1/2}\leq k\,\inj(M,g)$. Therefore $\gamma=\gamma_{\yy}$ for some $\yy\in\Lambda_k$, and the same holds for any other point belonging to the critical circle $S^1\cdot\gamma$. 
In order to remove this redundancy and retain only the  point $\gamma$ from its critical circle, we fix an arbitrary curve $\lambda\subset M$ that intersects $\gamma$ orthogonally at $\gamma(0)$, and consider the space
\[\Lambda_k(\lambda) :=\big\{ \xx\in\Lambda_k\ \big|\  x_0\in\lambda\big\}.\]
We denote by $E_k:\Lambda_k(\lambda)\to[0,\infty)$, $E_k(\xx):=E(\gamma_{\xx})$ the restricted energy functional. 
Notice that $\yy$ is a critical point of $E_k$, and it is an isolated point of $\crit(E_k)$ if and only if $S^1\cdot\gamma$ is isolated in $\crit(E)$ (see \cite[Prop.~3.1]{Asselle:2018aa}).

We recall that the Morse index $\ind(\gamma)$ is defined as the maximal dimension of a vector subspace $V\subset T_\gamma \Lambda$ such that $d^2E(\gamma)$ is negative definite on $V$. The nullity $\nul(\gamma)$ is defined as the dimension of the kernel of $d^2E(\gamma)$.
The Morse index $\ind(\yy)$ and the nullity $\nul(\yy)$ are defined analogously employing the Hessian $d^2E_k(\yy)$. It turns out that 
\begin{align*}
 \ind(\gamma)=\ind(\yy),\qquad
 \nul(\gamma)=\nul(\yy)+1.
\end{align*}
The difference $\nul(\gamma)-\nul(\yy)=1$ is due to the fact that $E_k$ is defined on the codimension one submanifold $\Lambda_k(\lambda)$ of $\Lambda_k$.
The kernel of the Hessian $d^2E(\gamma)$ is the vector space of 1-periodic Jacobi fields along $\gamma$, and such a space always contains the velocity vector field $\dot\gamma$. The kernel of the Hessian $d^2E_k(\yy)$, instead, is the vector space of those $\bm v\in T_{\yy}\Lambda_k(\lambda)$ whose associated $J_{\vv}$ is a 1-periodic orthogonal Jacobi field along $\gamma$, where orthogonal means that $g(J_{\vv},\dot\gamma)\equiv0$. In particular, $\yy$ is a non-degenerate critical point of $E_k$ if and only if $S^1\cdot\gamma$ is a non-degenerate critical circle of $E$, if and only if the unit-speed reparametrization of $\gamma$ is a non-degenerate $E(\gamma)^{1/2}$-periodic orbit of the geodesic flow as defined just before Definition~\ref{d:stable_link}.
We refer the reader to, e.g., \cite[Section~4]{Mazzucchelli:2016aa}), for the  proofs of these classical facts.

There is one last index of homological nature that can be associated to an isolated closed geodesic $\gamma=\gamma_{\yy}$. We  recall the construction in the setting $\Lambda_k(\lambda)$, and refer the reader to, e.g., \cite{Rademacher:1992aa,Bangert:2010aa,Asselle:2018aa} for more details. We equip $\Lambda_k(\lambda)$ with the Riemannian metric $g|_\lambda\oplus g\oplus...\oplus g$, and denote by $\theta^t$ the flow of the anti-gradient $-\nabla E_k$. 
Since we did not require the curve $\lambda$ introduced above to be closed, $\Lambda_k(\lambda)$ is not necessarily a complete Riemannian manifold and the flow lines of $\theta^t$ may not be defined for all time. This will not cause any issue, since we will only work locally near the critical point $\yy$. We will always tacitly slow down the orbits of $\theta^t$ away from a neighborhood of $\yy$ in which we work and assume that $\theta^t:\Lambda_k(\lambda)\to\Lambda_k(\lambda)$ is well defined for all $t\in\R$. We denote the flowout of a subset $W\subset\Lambda_k(\lambda)$ by 
\[\Theta(W)=\bigcup_{t\geq0} \theta^t(W).\] 
Moreover, we denote \[W^{<\ell}:=\big\{\xx\in W\ \big|\ E_k(\xx)<\ell^2\big\}.\] 
A neighborhood $W\subset \Lambda_k(\lambda)$ of $\yy$ is called a \emph{Gromoll-Meyer neighborhood} when 
\begin{itemize}
\setlength{\itemsep}{5pt}
 \item $W\cap\crit(E_k)=\{\yy\}$,
 \item $\Theta(W)\setminus W\subset \Lambda_k(\lambda)^{<\ell-\delta}$ for $\ell=E_k(\yy)^{1/2}$ and for some $\delta>0$.
\end{itemize}
This notion can be associated to isolated critical points of arbitrary functions, and a simple argument from Morse theory shows that any such critical point admits an arbitrarily small Gromoll-Meyer neighborhood.
The \emph{local homology} of $\yy$ is the relative homology group
\begin{align*}
 C_*(\yy):=H_*(W,W^{<\ell-\delta}).
\end{align*}
The notation suggests that the local homology is independent of the choice of the Gromoll-Meyer neighborhood and of the small constant $\delta>0$, as can be easily showed by means of a deformation argument with the anti-gradient flow $\theta^t$.
Since $\gamma$ is an isolated closed geodesic, and thus $\yy$ is an isolated critical point of $E_k$, the local homology $C_*(\yy)$ is finitely generated \cite[page 364]{Gromoll:1969aa}. Furthermore, if $\gamma$ is non-degenerate, the Morse lemma \cite[Lemma~2.2]{Milnor:1963aa} implies that 
\begin{align}
\label{e:loc_hom_1}
 C_d(\yy)
 \cong
\left\{
  \begin{array}{@{}ll}
    \Z, & \mbox{if } d=\ind(\gamma), \vspace{5pt}\\ 
    0, & \mbox{if } d\neq\ind(\gamma).
  \end{array}
\right.
\end{align}

\begin{figure}
\begin{footnotesize}
\includegraphics{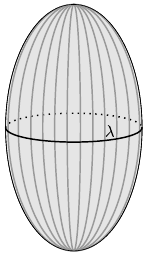}
\end{footnotesize}
\caption{The family of meridians (in gray) in a 2-sphere of revolution, intersecting the equator $\lambda$ orthogonally.}
\label{f:meridians}
\end{figure}%

In the proof of Theorem~\ref{mt:sphere}, we will also need to consider non-isolated closed geodesics, and more precisely the meridians of a 2-sphere of revolution (Figure~\ref{f:meridians}). More generally, the setting is the following. Let $\lambda$ be an embedded circle in the closed Riemannian surface $(M,g)$, and $Z\subset\crit^+(E)$ be a circle of closed geodesics such that each $\gamma\in Z$ intersects $\lambda$ orthogonally at $\gamma(0)$. Notice that $S^1\cdot Z\subset\crit^+(E)$ is a critical torus, and assume that it is isolated in $\crit^+(E)$. We denote by $\ell:=E(Z)^{1/2}$ the corresponding length, and choose the integer $k$ to be large enough so that $\ell<k\,\inj(M,g)$. The circle 
\[Z_k:=\big\{\xx\in\Lambda_k(\lambda)\ \big|\ \gamma_{\xx}\in Z\big\}\]
is a connected component of $\crit(E_k)$, and admits an arbitrarily small Gromoll Meyer neighborhood $W\subset\Lambda_k(\lambda)$, which is a neighborhood satisfying $W\cap\crit(E)=Z_k$ and $\Theta(W)\setminus W\subset \Lambda_k(\lambda)^{<\ell-\delta}$ for some $\delta>0$. The local homology of $Z_k$ is the relative homology group 
\begin{align*}
 C_*(Z_k):=H_*(W,W^{<\ell-\delta}).
\end{align*}
Assume that $Z_k$ is a non-degenerate critical circle, meaning that
\[
\ker(d^2E_k(\xx))=T_{\xx}Z_k,\qquad\forall \xx\in Z_k.
\]
Equivalently, $\nul(\xx)=\dim(Z_k)$. Let $N\to Z_k$ be the negative bundle, which is the vector bundle whose fibers $N_{\xx}$ consist of the direct sum of the negative eigenspaces of $d^2E_k(\xx_s)$. Let $0_N\subset N$ be the zero-section of $N$. By Morse-Bott lemma \cite[Lemma~3.51]{Banyaga:2004aa}, the local homology of $Z_k$ is is given by
\begin{align}
\label{e:loc_hom_2}
C_*(Z_k) \cong H_*(N,N\setminus0_N).
\end{align}
Notice that all critical points $\xx\in Z_k$ have the same Morse index $\ind(Z_k):=\ind(\xx)$, which is the rank of $N$.
If the negative bundle $N$ is orientable, Thom isomorphism theorem implies $H_*(N,N\setminus0_N)\cong H_{*-\ind(Z_k)}(Z_k)$, and since $Z_k$ is a circle we conclude
\begin{align}
\label{e:loc_hom_3}
C_d(Z_k)  
 \cong
\left\{
  \begin{array}{@{}ll}
    \Z, & \mbox{if } d\in\{\ind(Z_k),\ind(Z_k)+1\}, \vspace{5pt}\\ 
    0, & \mbox{if }d\not\in\{\ind(Z_k),\ind(Z_k)+1\}.
  \end{array}
\right.
\end{align}

\section{Morse theory within a flat knot type}\label{s:Morse}

\subsection{Neighborhoods of compact subsets of closed geodesics}\label{ss:nbhds}

Let $(M,g)$ be a closed  surface, $\bm\zeta=(\zeta_1,...,\zeta_n)$ be either the empty link (when $n=0$) or a flat link of closed geodesics, and $\KK$ a primitive flat knot type relative to $\bm\zeta$. We denote by
\begin{align*}
 \Gamma(\KK)=\Gamma_g(\KK)
\end{align*}
the subset of closed geodesics in $\KK$. 
We define the \emph{$\KK$-spectrum}  
\begin{align*}
\sigma(\KK) = \sigma_g(\KK) := \big\{ L(\gamma)\ \big|\ \gamma\in\Gamma(\KK) \big\}. 
\end{align*}
The variational characterization of closed geodesics, together with Sard theorem, implies that $\sigma(\KK)$ is a closed subset of $\R$ of zero Lebesgue measure.
In the next subsection we shall introduce the notion of local homology of a compact set of closed geodesic of a given length in the setting of the curve shortening flow. As a preliminary step, in this subsection we extend to this setting the notion of Gromoll-Meyer neighborhood \cite{Gromoll:1969aa,Gromoll:1969ab}, already encountered in Section~\ref{ss:energy}.

We recall that, for each $\gamma\in\Omega$, the curve shortening flow trajectory $t\mapsto\phi^t(\gamma)$ is defined on the maximal interval $[0,t_\gamma)$. We set
\begin{align}
\label{e:exit_time_K}
 t_{\gamma,\KK}:=t_{g,\gamma,\KK}:=\sup\big\{ t\in(0,t_{\gamma})\ \big|\ \phi^t(\gamma)\in\KK \big\},
\end{align}
so that $[0,t_{\gamma,\KK})$ is the maximal interval such that the trajectory $t\mapsto\phi^t(\gamma)$ stays in $\KK$.
For each subset $\mathcal Y\subset\KK$, we denote its flowout in $\KK$ under the curve shortening flow by 
\begin{align*}
 \Phi(\mathcal Y)
 =
 \Phi_g(\mathcal Y)
 :=
 \big\{ \phi^t(\gamma)\ \big|\ \gamma\in \mathcal Y,\ t\in[0,t_{\gamma,\KK}) \big\}.
\end{align*} 
We recall that $\Omega$ is endowed with the quotient $C^3$ topology. Since we will also consider subsets that are open in the coarser $C^2$ topology, we will always specify whether a subset is $C^3$-open or $C^2$-open.

\begin{Lemma}\label{l:encapsulated_nbhds}
For each compact interval $[a,b]\subset(0,\infty)$ and for each $C^2$-open neighborhood $\UU$ of $\Gamma(\KK)\cap L^{-1}([a,b])$ there exist $\delta>0$ and a $C^3$-open neighborhood $\VV\subset\UU$ of $\Gamma(\KK)\cap L^{-1}([a,b])$ such that, whenever $\gamma\in\VV$ and $\phi^t(\gamma)\not\in\UU$ for some $t\in(0,t_\gamma)$, we have $L(\gamma)-L(\phi^t(\gamma))\geq\delta$.
\end{Lemma}

\begin{proof}
For each $\gamma\in\Omega$, we denote by $K_\gamma$ the signed geodesic curvature of its arclength parametrization with respect to a normal vector field, as in Section~\ref{ss:curvature}.
The $C^3$ topology on $\Omega$ guarantees that the function $\gamma\mapsto\|\dot K_\gamma\|_{L^\infty}$ is continuous.
The family of subsets
\[
\UU(\epsilon):=\big\{ \gamma\in\KK\ \big|\ L(\gamma)\in(a-\epsilon,b+\epsilon),\ \|K_\gamma\|_{L^\infty}<\epsilon \big\},\ \epsilon>0,
\]
is a fundamental system of $C^2$-open neighborhoods of the compact set of closed geodesics $\Gamma(\KK)\cap L^{-1}([a,b])$. For each $\delta>0$, the subset
\[
\VV(\delta):=\big\{ \gamma\in\KK\ \big|\ L(\gamma)\in(a-\delta,b+\delta),\ \|K_\gamma\|_{L^2}+\|\dot K_\gamma\|_{L^2}<\delta \big\}
\]
is a $C^3$-open neighborhood of $\Gamma(\KK)\cap L^{-1}([a,b])$. Therefore, the statement is a direct consequence of Proposition~\ref{p:curvature_control}.
\end{proof}

We recall that, just before its limit time defined in~\eqref{e:exit_time_K}, any orbit of the curve shortening flow starting in $\KK$ reaches the subsets $\overline\KK_\rho$ defined in \eqref{e:exit_and_subloops}, for $\rho>0$ arbitrarily small.
For each subset $\WW\subset\Omega$ and $b\in(0,\infty)$, we denote
\begin{align*}
 \WW^{<b}=\WW_g^{<b}:=\big\{\gamma\in\WW\ \big|\ L(\gamma)<b\big\}.
\end{align*}
Consider a spectral value $\ell\in\sigma(\KK)$, and a connected component $Z$ of the compact subset of closed geodesics $\Gamma(\KK)\cap L^{-1}(\ell)$. 
We define a \emph{Gromoll-Meyer neighborhood} of $Z$ to be a (not necessarily open) neighborhood $\VV\subset\KK$ of $Z$ such that
\begin{itemize}
\setlength{\itemsep}{5pt}

\item $\VV\cap\overline\KK_\rho=\varnothing$ for some $\rho>0$, 

\item $\VV\cap\Gamma(\KK)\cap L^{-1}(\ell)=Z$, 

\item $\Phi(\VV)\setminus\VV\subset\KK^{<\ell-\delta}$ for some $\delta>0$.

\end{itemize}
While in standard Morse theoretic settings such as in Section~\ref{ss:energy} one can always construct arbitrarily small open Gromoll-Meyer neighborhoods, in our setting of the curve shortening flow we can only insure the existence of arbitrarily $C^2$-small Gromoll-Meyer neighborhoods, essentially as a consequence of Lemma~\ref{l:encapsulated_nbhds}. We stress that Gromoll-Meyer neighborhoods are not $C^3$-open, but of course they must contain a $C^3$-open neighborhood of $Z$.

\begin{Lemma}\label{l:GM_nbhds}
Any connected component of $\Gamma(\KK)\cap L^{-1}(\ell)$ admits an arbitrarily $C^2$-small Gromoll-Meyer neighborhood.
\end{Lemma}

\begin{proof}
Let $Z$ be a connected component of $\Gamma(\KK)\cap L^{-1}(\ell)$, and we set $Z':=\Gamma(\KK)\cap L^{-1}(\ell)\setminus Z$ to be the union of the remaining connected components. We consider two arbitrarily small disjoint $C^2$-open neighborhoods $\UU$ and $\UU'$ of $Z$ and $Z'$ respectively. By Lemma~\ref{l:encapsulated_nbhds}, there exists $\delta>0$, a $C^3$-open neighborhood $\VV\subset\UU$ of $Z$, and a $C^3$-open neighborhood $\VV'$ of  $Z'$ such that, whenever $\gamma\in\VV\cup\VV'$ and $\phi^t(\gamma)\not\in\UU\cup\UU'$ for some $t\in(0,t_\gamma)$, we have $L(\gamma)-L(\phi^t(\gamma))\geq2\delta$. The intersection 
\[
\WW:=\Phi(\VV)\cap L^{-1}(\ell-\delta,\ell+\delta)
\]
is a Gromoll-Meyer neighborhood of $Z$ contained in $\UU$.
\end{proof}

\subsection{Local homology}\label{ss:hom_visible}
Let $\ell\in\sigma(\KK)$ be a spectral value, and $Z$ a connected component of $\Gamma(\KK)\cap L^{-1}(\ell)$. We assume that $Z$ is an \emph{isolated} family of closed geodesics, that is, it admits a neighborhood $\UU\subset\KK$ such that $\UU\cap\Gamma(\KK)=Z$.
The \emph{local homology} of $Z$ is the relative homology group with integer coefficients
\begin{align*}
C_*(Z)
:=
H_*(\UU\cup\KK^{<\ell-\delta},\KK^{<\ell-\delta}),
\end{align*}
where $\UU$ is a Gromoll-Meyer neighborhood of $Z$ such that $\UU\cap\Gamma(\KK)=Z$, and $\delta\geq0$ is small enough. By a simple deformation argument employing the curve shortening flow and Lemma~\ref{l:encapsulated_nbhds}, one readily sees that $C_*(Z)$ is independent of the choice of  $\UU$ and  $\delta$ (we stress that this is the case also for $\delta=0$). Moreover,  $C_*(Z)$ depends only on the Riemannian metric $g$ in a neighborhood of $Z$. More precisely, for any $C^2$-open neighborhood $\VV\subset\KK$ of $Z$, we can choose the Gromoll-Meyer neighborhood $\UU$ to be contained in $\VV$, and by the excision property of singular homology  the inclusion induces an isomorphism
\begin{align*}
 H_*(\UU\cup\VV^{<\ell-\delta},\VV^{<\ell-\delta})
 \ttoup^{\cong}
 C_*(Z).
\end{align*}
The family of closed geodesics $Z$ is called \emph{homologically visible} when it is isolated and has non-trivial local homology.

In this paper, we will particularly need to consider two kind of isolated compact sets of closed geodesics $Z\subset\Gamma(\KK)$. With an abuse of notation, we will describe them them as families of parametrized closed geodesics in the set of critical points $\crit^+(E)\subset\Lambda$, in the setting of  Section~\ref{ss:energy}.

\begin{itemize}
\setlength{\itemsep}{5pt}

\item[(i)] $Z\subset\crit^+(E)\cap E^{-1}(\ell^2)$ is a circle of closed geodesics, and each $\gamma\in Z$ intersects a simple closed geodesic $\lambda\subset M$ orthogonally at $\gamma(0)$ (as in the example of Figure~\ref{f:meridians}).

\item[(ii)] $Z=\{\gamma\}\subset\crit^+(E)\cap E^{-1}(\ell^2)$ is a singleton. In this case, we set $\lambda\subset M$ to be any open geodesic segment intersecting $\gamma$ orthogonally at $\gamma(0)$.

\end{itemize}
We recall that all the loops in $\Omega$ are oriented. In this subsection, on each $\zeta\in\KK$ sufficiently $C^2$-close to some $\gamma\in Z$, we fix the unique parametrization $\zeta:S^1\looparrowright M$ with constant speed $\|\dot\zeta\|_g\equiv L(\zeta)$ and such that $\zeta(0)$ is  an intersection point $\zeta\cap\lambda$; there may be more than one such intersection point, but we choose the unique one that makes $\zeta$ $C^2$-close to $\gamma$ parametrized as in points (i) and (ii).

We fix an integer $k>\ell/\inj(M,g)$, and consider the space of broken geodesic loops $\Lambda_k(\lambda)$ and the energy functional $E_k:\Lambda_k(\lambda)\to[0,\infty)$ introduced in Section~\ref{ss:energy}. We set
\begin{align*}
Z_k:=\big\{\xx\in\Lambda_k(\lambda)\ \big|\ \gamma_{\xx}\in Z \big\},
\end{align*}
We already introduced the local homology $C_*(Z_k)$ and pointed out that it can be studied by means of classical Morse theory. In particular, we know that $C_*(Z_k)$ is always finitely generated, and when $Z_k$ is a non-degenerate critical manifold of $E_k$ it is given by~\eqref{e:loc_hom_1} and~\eqref{e:loc_hom_2}. While certain methods of local Morse theory are not directly available in the setting $\Omega$ of unparametrized oriented immersed loops, we infer analogous properties for the local homology $C_*(Z)$ from those of $C_*(Z_k)$.

\begin{Lemma}\label{l:finitely_generated}
The local homology $C_*(Z)$ is isomorphic to a subgroup of $C_*(Z_k)$. In particular, $C_*(Z)$ is finitely generated.
\end{Lemma}

\begin{proof}
In the proof, we will need to regularize loops near the compact space of closed geodesics $Z$. This can be done by convolution, as follows. We embed $M$ into an Euclidean space $\R^n$, and consider a tubular neighborhood $N\subset\R^n$ of $M$ with associated smooth projection $\pi:N\to M$. Let $\chi:\R\to[0,\infty)$ be a smooth function supported in $[-1,1]$ and whose integral is 1. For $\epsilon>0$, the family of functions 
$\chi_\epsilon:\R\to[0,\infty)$, $\chi_\epsilon(t)=\epsilon^{-1}\chi(\epsilon^{-1}t)$ 
tend to the Dirac delta at the origin as $\epsilon\to0$. On a sufficiently small neighborhood $\XX\subset\Lambda$ of $Z$, for each $\epsilon>0$ small enough we have a continuous map
\begin{align*}
f_\epsilon:\XX \to \Lambda, 
\end{align*}
such that $f_\epsilon(\gamma)$ is the constant speed reparametrization of the loop $\pi(\gamma*\chi_\epsilon)$, the symbol $*$ denoting the convolution operation. We extend the family $f_\epsilon$ continuously at $\epsilon=0$ by setting $f_0:\XX \hookrightarrow \Lambda$ to be the inclusion.
Notice that
\begin{align*}
 L(f_\epsilon(\gamma))^2
 = 
 E(f_\epsilon(\gamma))
 \leq 
 E(\pi(\gamma*\chi_\epsilon)).
\end{align*}
By means of a partition of unity we can construct a continuous function 
\[\overline\epsilon:\XX\times[0,\infty)\to[0,\infty),\] 
such that $\overline\epsilon(\cdot,0)\equiv0$, and for each $\delta>0$ the values $\overline\epsilon(\gamma,\delta)>0$ are positive and small enough such that the following points hold:
\begin{itemize}
\setlength{\itemsep}{5pt}

\item[(i)] $\sqrt{E(f_\epsilon(\gamma))}\leq \sqrt{E(\gamma)}+\delta/2$ for all $\gamma\in\XX$ and $\epsilon\in[0,\overline\epsilon(\gamma,\delta)]$.

\item[(ii)] For a small enough $C^2$-open neighborhood $\YY\subset\KK\cap\XX$ of $Z$, 
consider the continuous homotopy 
\[r_s:\YY\to\XX,\ r_s(\gamma)=\gamma_s,\qquad s\in[0,1],\] 
were $\gamma_s$ is defined as follows: for each $i\in\Z_k$, we set
\[\gamma_s|_{[i/k,(i+1-s)/k]}:=\gamma|_{[i/k,(i+1-s)/k]},\]
whereas $\gamma_s|_{[(i+1-s)/k,(i+1)/k]}$ is defined as the shortest geodesic segment joining its endpoints. Notice that $r_0$ is the inclusion, $r_1$ takes values inside $\Lambda_k(\lambda)$ seen as a subset of broken geodesic loops in $\Lambda$, and
\[E\circ r_s(\gamma)\leq E(\gamma)= L(\gamma)^2.\]
We require $\YY$ to be sufficiently $C^2$-small so that
\[
f_\epsilon\circ r_s(\gamma)\in\KK,\qquad\forall \gamma\in\YY,\ s\in[0,1],\ \epsilon\in(0,\overline\epsilon(\gamma,\delta)].
\]

\item[(iii)] If $Y\subset\Lambda_k(\lambda)\cap\XX$ is a small enough neighborhood of $Z_k$, we have
\begin{align*}
 f_\epsilon(\gamma_{\xx})\in\KK,\qquad\forall\xx\in Y,\ \epsilon\in(0,\overline\epsilon(\gamma_{\xx},\delta)].
\end{align*}
\end{itemize}
We define the family of continuous maps
\begin{align*}
h_\delta:\XX\to\Lambda,
\qquad h_\delta(\gamma)=f_{\overline\epsilon(\gamma,\delta)}(\gamma).
\end{align*}
Notice that $h_0$ is the inclusion.
As usual, we write $h_\delta(\xx)=h_\delta(\gamma_{\xx})$ for all $\xx\in Y$.
We fix:
\begin{itemize}
\setlength{\itemsep}{5pt}

\item a Gromoll-Meyer neighborhood $W\subset Y$ of $Z_k$,

\item a $C^2$-open neighborhood $\VV_1\subset\YY$ of $Z$ that is small enough so that 
\[r_1(\VV_1)\subset W,\]

\item a $C^2$-open neighborhood $\VV_2\subset\KK\setminus\VV_1$ of the compact set of closed geodesics 
\[\Gamma:=\Gamma(\KK)\setminus Z\cap L^{-1}(\ell),\]

\item a Gromoll-Meyer neighborhood $\UU_1\subset\VV_1$ of $Z$,

\item a subset $\UU_2\subset\VV_2$ that is the union of Gromoll-Meyer neighborhoods of the connected components of $\Gamma$,

\item $\delta>0$ small enough as in the definition of local homology for the Gromoll-Meyer neighborhoods introduced thus far, and such that 
\[\Gamma(\KK)\cap L^{-1}[\ell-\delta,\ell+\delta]\subset\UU_1\cup\UU_2.\]
\end{itemize}

We replace $\UU_1$, $\UU_2$, and $W$ with $\UU_1^{<\ell+\delta/2}$, $\UU_2^{<\ell+\delta}$, and $W^{<\ell+\delta/2}$ respectively, so that in particular 
\[\UU:=\UU_1\cup\UU_2\subset L^{-1}[\ell-\delta,\ell+\delta),
\qquad
W\subset\Lambda_k(\lambda)^{<\ell+\delta/2}.\] 
We set $\VV:=\VV_1\cup\VV_2$. We also consider the subspaces $\overline K_\rho$ introduced in~\eqref{e:exit_and_subloops}, and we fix $\rho\in(0,\rho_g]$ small enough so that 
\[\overline\KK_\rho\cap\UU=\varnothing.\]
By excision, the inclusion induces a homology isomorphism
\begin{align*}
H_*(\UU\cup\VV^{<\ell-\delta/2},\VV^{<\ell-\delta/2})
\ttoup^{\cong}
H_*(\UU\cup\overline\KK{}^{<\ell-\delta/2}\cup\overline\KK_\rho,\overline\KK{}^{<\ell-\delta/2}\cup\overline\KK_\rho).
\end{align*}
Moreover, the inclusion
\[
\UU\cup\overline\KK{}^{<\ell-\delta/2}\cup\overline\KK_\rho
\hookrightarrow
\overline\KK{}^{<\ell+\delta}\cup\overline\KK_\rho
\]
is a homotopy equivalence, whose homotopy inverse can be built by pushing with the curve shortening flow. Overall, we infer that the inclusion induces a homology isomorphism
\begin{align*}
H_*(\UU\cup\VV^{<\ell-\delta/2},\VV^{<\ell-\delta/2})
\ttoup^{\cong}
H_*(\overline\KK{}^{<\ell+\delta}\cup\overline\KK_\rho,\overline\KK{}^{<\ell-\delta/2}\cup\overline\KK_\rho).
\end{align*}
Since $\VV$ is the disjoint union of $\VV_1$ and $\VV_2$, we have
\[
H_*(\UU\cup\VV^{<\ell-\delta/2},\VV^{<\ell-\delta/2})
\cong
H_*(\UU_1\cup\VV_1^{<\ell-\delta/2},\VV_1^{<\ell-\delta/2})
\oplus
H_*(\UU_2\cup\VV_2^{<\ell-\delta/2},\VV_2^{<\ell-\delta/2}),
\]
and therefore we infer  that the inclusion induces an injective homomorphism
\begin{align*}
H_*(\UU_1\cup\VV_1^{<\ell-\delta/2},\VV_1^{<\ell-\delta/2})
\eembup
H_*(\overline\KK{}^{<\ell+\delta}\cup\overline\KK_\rho,\overline\KK{}^{<\ell-\delta/2}\cup\overline\KK_\rho).
\end{align*}
Notice that the inclusion induces an isomorphism
\begin{align*}
 H_*(\UU_1\cup\VV_1^{<\ell-\delta},\VV_1^{<\ell-\delta})
 \ttoup^{\cong}
 H_*(\UU_1\cup\VV_1^{<\ell-\delta/2},\VV_1^{<\ell-\delta/2}),
\end{align*}
and therefore an injective homomorphism
\begin{align*}
H_*(\UU_1\cup\VV_1^{<\ell-\delta},\VV_1^{<\ell-\delta})
\eembup
H_*(\overline\KK{}^{<\ell+\delta}\cup\overline\KK_\rho,\overline\KK{}^{<\ell-\delta/2}\cup\overline\KK_\rho).
\end{align*}

By points (i) and (iii), we have 
\[
h_\delta(W)\subset\KK^{<\ell+\delta},\qquad
h_\delta(W^{<\ell-\delta})\subset\KK^{<\ell-\delta/2}.
\]
Moreover, by points (i) and (ii), we can construct a continuous homotopy
\begin{align*}
j_s:
(\UU_1\cup\VV_1^{<\ell-\delta},\VV_1^{<\ell-\delta})
\toup
(\overline\KK{}^{<\ell+\delta}\cup\overline\KK_\rho,\overline\KK{}^{<\ell-\delta/2}\cup\overline\KK_\rho),\quad s\in[0,1],
\end{align*}
given by
\[
j_s(\gamma)
=
\left\{
  \begin{array}{@{}ll}
    h_{2s\delta}(\gamma), & \mbox{if }s\in[0,1/2], \vspace{5pt}\\ 
    h_\delta\circ r_{2s-1}(\gamma), & \mbox{if }s\in[1/2,1]. \\ 
  \end{array}
\right.
\]
The map $j_0$ is the inclusion, and $j_1=h_\delta\circ r_1$. Overall, we obtain a commutative diagram
\[
\begin{tikzcd}[row sep=large]
C_*(Z)\cong H_*(\UU_1\cup\VV_1^{<\ell-\delta},\VV_1^{<\ell-\delta}) 
 \arrow[r, "r_{1*}"]
 \arrow[dr, hook] 
 &
 H_*(W,W^{<\ell-\delta}) = C_*(Z_k) \arrow[d,"h_{\delta*}"]\\
 & 
 H_*(\overline\KK{}^{<\ell+\delta}\cup\overline\KK_\rho,\overline\KK{}^{<\ell-\delta/2}\cup\overline\KK_\rho) 
\end{tikzcd}
\]
and we infer that $r_{1*}$ is injective.
\end{proof}

We recall that $Z\subset\Gamma(\KK)$ is originally an unpararametrized family of closed geodesics, but we fixed the parametrization with constant speed on each $\gamma\in Z$ so that $\gamma(0)\in\lambda$. Equipped with these parametrizations, $Z$ belongs to an isolated critical manifold $S^1\cdot Z\subset\crit(E)$, where the circle acts on $Z$ by time translation (see Section~\ref{ss:energy}). We say that $Z\subset\Gamma(\KK)$ is \emph{non-degenerate} when the corresponding $S^1\cdot Z$ is a non-degenerate critical manifold of $E$. Equivalently, $Z_k$ is a non-degenerate critical manifold of $\crit(E_k)$, i.e.
\[
\ker(d^2E_k(\xx))=T_{\xx}Z_k,
\qquad
\forall \xx\in Z_k.
\]
We denote by $N\to Z_k$ the negative bundle, which is the vector bundle whose fibers $N_{\xx}$ are the the direct sum of the negative eigenspaces of the Hessian $d^2E_k(\xx)$. We denote by $0_N\subset N$ its zero-section.  In~\eqref{e:loc_hom_1} and~\eqref{e:loc_hom_2}, we showed that the local homology $C_*(Z_k)$ is fully determined by the Morse index $\ind(Z_k)$, at least when   the negative bundle $N$ is orientable (which is trivially satisfied if $Z$ consists of a single closed geodesic as in point (ii)). We derive the same conclusion for the local homology $C_*(Z)$.

\begin{Lemma}\label{l:non_deg_visible}
If $Z$ is non-degenerate, its local homology is given by
\begin{align*}
C_*(Z)\cong H_*(N,N\setminus0_N).
\end{align*}
In particular, if $Z=\{\gamma\}$ consists of a single  closed geodesic as in point $(ii)$, we have
\[
C_d(\gamma)
\cong
\left\{
  \begin{array}{@{}ll}
    \Z, & \mbox{if } d=\ind(\gamma), \vspace{5pt}\\ 
    0, & \mbox{if } d\neq\ind(\gamma). 
  \end{array}
\right.
\]
If instead $Z$ consists of a circle of non-degenerate closed geodesics as in point $(i)$, and the negative bundle $N\to Z_k$ is orientable, we have
\begin{align}
\label{e:loc_hom_circle}
C_d(Z)  
 \cong
\left\{
  \begin{array}{@{}ll}
    \Z, & \mbox{if } d\in\{\ind(Z_k),\ind(Z_k)+1\}, \vspace{5pt}\\ 
    0, & \mbox{if }d\not\in\{\ind(Z_k),\ind(Z_k)+1\}.
  \end{array}
\right.
\end{align}
\end{Lemma}

\begin{proof}
A negative bundle for $Z$ is any vector sub-bundle $R\to Z$ of $T\Lambda|_Z$ of rank $\ind(Z)$ and such that 
\[d^2E(\gamma)[Y,Y]<0,\qquad\forall \gamma\in Z,\ Y\in R_\gamma\setminus\{0\}.\]
We recall that $\Lambda_k$ embeds as a subspace of 1-periodic broken geodesics in $\Lambda$. Under this embedding, a tangent vector $\vv\in T_{\xx}\Lambda_k$ corresponds to a 1-periodic broken Jacobi field $J_{\vv}$ along $\gamma_{\xx}$ such that $J_{\vv}(0)=v_0$ is tangent to the geodesic $\lambda$. 

Therefore, we can see the vector bundle $N\to Z_k$ as a negative bundle $N\to Z$, so that each fiber $N_{\gamma}$ is a vector space of dimension $\ind(Z)$ containing 1-periodic continuous broken Jacobi fields $Y$ along $\gamma$ such that $g(Y(0),\dot\gamma(0))=0$ (since $Y(0)$ is tangent to the geodesic $\lambda$).
While these vector fields are not smooth, we can slightly modify the vector bundle $N\to Z$ to make them smooth while preserving the negative definiteness of the Hessian of the energy. This can be done by convolution, similarly as in the proof of Lemma~\ref{l:finitely_generated}, which provides for all $\epsilon\geq0$ small enough a family of injective bundle homomorphisms
\begin{align*}
f_\epsilon:N\hookrightarrow T\Lambda|_{Z},
\end{align*}
which depends continuously on the parameter $\epsilon$. The map $f_0$ is simply the inclusion, whereas for each $\epsilon>0$ all vector fields in the image of $f_\epsilon$ are smooth. We fix $\epsilon>0$ small enough so that the vector bundle $P:=f_\epsilon(N)\to Z$ is a negative bundle for $Z$. We recall that, for each $\gamma\in Z$, the vector field $\dot\gamma$ belongs to the kernel of $d^2E(\gamma)$. Therefore, we have another injective bundle homomorphism
\begin{align*}
h:P\hookrightarrow T\Lambda|_Z,\qquad h(Y)=Y-g(Y(0),\dot\gamma(0))\dot\gamma.
\end{align*}
The vector bundle $Q:=h(P)\to Z$ is a negative bundle for $Z$, and each fiber $Q_\gamma$ consists of smooth 1-periodic vector fields $Y$ such that $g(Y(0),\dot\gamma(0))=0$. Notice that $N\to Z$ and $Q\to Z$ are isomorphic vector bundles.

We fix a Gromoll-Meyer neighborhood $W\subset\Lambda_k(\lambda)$ of $Z_k$, and a $C^2$-open neighborhood $\VV\subset\KK$ of $Z$ that is small enough so that we have a well defined continuous map
\begin{align*}
r:\VV\to W,\qquad r(\zeta)=(\zeta(0),\zeta(1/k),...,\zeta((k-1)/k)).
\end{align*}

We fix an open tubular neighborhood $B\subset Q$ of the zero-section $0_Q$, and require $B$ to be small enough so that we have a smooth embedding $i:B\hookrightarrow \Lambda$ given by
\begin{align*}
i(\xi)(t)=\exp_{\gamma(t)}(\xi(t)),\qquad\forall\gamma\in Z,\ \xi\in Q_\gamma,
\end{align*}
where $\exp$ denotes the Riemannian exponential map. Notice that $d i(0)\xi=\xi$, and therefore the restriction of $E\circ i$ to each fiber $B_\gamma$ has a non-degenerate local maximum at the origin. 
Moreover, $i(\xi)(0)\in\lambda$, and therefore $r\circ i(\xi)\in\Lambda_k(\lambda)$.
Since $E_k\circ r\circ i(\xi)\leq E\circ i(\xi)$ for all $\xi\in B$ and $r\circ i(0)= i(0)=\gamma$, we infer that the restriction of $E_k\circ r\circ i$ to each fiber $B_\gamma$ also has a non-degenerate local maximum at the origin, and that $d(r\circ i)(0)=dr(\gamma)$ is injective. Up to shrinking the tubular neighborhood $B$, we have that $r\circ i:B\hookrightarrow \Lambda_k(\lambda)$ is an embedding.

We consider a smaller tubular neighborhood $B'\subset \overline{B'}\subset B$ of the zero-section $0_Q$, and $\delta>0$ be small enough so that
\[
\sup_{B\setminus B'} E\circ i<\ell-\delta.
\]
Up to reducing $\delta$, the Morse-Bott lemma \cite[Lemma~3.51]{Banyaga:2004aa} readily implies that the composition $r\circ i$ induces an isomorphism
\begin{align}
\label{e:ri_isom}
 (r\circ i)_*:H_*(B,B')\ttoup^{\cong} H_*(W,W^{<\ell-\delta}).
\end{align}
Let $\UU\subset\VV$ be a Gromoll-Meyer neighborhood of $Z$. Up to further reducing $\delta$, the  isomorphism \eqref{e:ri_isom} factors as in the following commutative diagram.
\[
\begin{tikzcd}[row sep=large]
 H_*(B,B') 
 \arrow[rrr, "i_*"]
 \arrow[ddrrr, "(r\circ i)_*"', "\cong"] 
 & & &
 H_*(\UU\cup\VV^{<\ell-\delta},\VV^{<\ell-\delta}) \cong C_*(Z)
 \arrow[dd,"r_*"]\\\\
 & & &
 H_*(W,W^{<\ell-\delta})=C_*(Z_k) 
\end{tikzcd}
\]
This implies that $r_*$ is surjective, and therefore
\begin{align*}
\rank(C_d(Z))\geq \rank(C_d(Z_k)),\qquad\forall d\geq0.
\end{align*}
Moreover, by Lemma~\ref{l:finitely_generated}, $C_d(Z)$ is  isomorphic to a subgroup of $C_d(Z_k)$ in each degree $d$. Therefore, we have $C_*(Z)\cong C_*(Z_k)$ provided $C_d$ is either trivial or isomorphic to $\Z$ in each degree $d$. 

This holds under our assumptions. Indeed, in case (ii), $Z=\{\gamma\}$ consists of a single closed geodesic $\gamma=\gamma_{\xx}$, and therefore $Z_k=\{\xx\}$ consists of a single non-degenerate critical point; the local homology $C_d(\xx)=C_d(Z_k)$ is isomorphic to $\Z$ in degree $d=\ind(\gamma)$, and vanishes in all the other degrees $d$ (Equation~\eqref{e:loc_hom_1}). 
In case (i), $Z_k$ is a non-degenerate critical circle; If the negative bundle $N\to Z_k$ is orientable, the local homology $C_*(Z_k)$ is isomorphic to $\Z$ in degrees $d=\ind(Z_k)$ and $d=\ind(Z_k)+1$, and vanishes in all the other degrees $d$ (Equation~\eqref{e:loc_hom_3}). 
\end{proof}

\subsection{Global Morse theory}\label{ss:global}

In his seminal work \cite[Theorem~1.1]{Angenent:2005aa}, Angenent managed to frame the curve shortening flow in the setting of Morse-Conley theory \cite{Conley:1978aa}, and in particular proved that a primitive relative flat knot type $\KK$ contains a closed geodesic provided the quotient $\overline\KK/\overline\KK_\rho$ is not contractible. Here, $\overline\KK_{\rho}\subset\overline\KK$ is the subset defined in~\eqref{e:exit_and_subloops}, for any $\rho\in(0,\rho_g]$, where $\rho>0$ is given by Lemma~\ref{l:subloops}(ii).
In this section, we provide more results on global Morse theory within the primitive relative flat knot type $\KK$, which will be employed in the proofs of Theorems~\ref{mt:multiplicity} and~\ref{mt:sphere}.

The filtration 
$\overline\KK^{<\ell}\cup\overline\KK_\rho$, for $\ell>0$,
together with the curve shortening flow, implies:
\begin{itemize}
\setlength{\itemsep}{5pt}

\item[(i)] If $\Gamma(\KK)\cap L^{-1}[a,b)=\varnothing$, then the inclusion
\begin{align}
\label{e:inclusion_sublevel_sets_within_K}
 \overline\KK^{<a}\cup\overline\KK_\rho
 \hookrightarrow
 \overline\KK^{<b}\cup\overline\KK_\rho
\end{align}
 is a homotopy equivalence, and in particular induces a homology isomorphism.
 
\item[(ii)] If $\sigma(\KK)\cap (a,b)=\{\ell\}$, we have an isomorphism
\begin{align*}
 H_*(\overline\KK^{<b}\cup\overline\KK_\rho,\overline\KK^{<a}\cup\overline\KK_\rho)
 \cong
 \bigoplus_{Z}
 C_*(Z),
\end{align*}
where the direct sum ranges over the connected components $Z$ of $\Gamma(\KK)$.

\end{itemize}

We define the \emph{local homology} of the primitive relative flat knot type $\KK$ as the relative homology group
\begin{align*}
C_*(\KK):=H_*(\overline\KK,\overline\KK_\rho). 
\end{align*}
Angenent's Lemma~\ref{l:subloops} readily implies that $C_*(\KK)$ is independent of the choice of $\rho\in(0,\rho_g]$ and of the admissible Riemannian metric $g$, where admissible means that $\bm\zeta$ is a flat link of closed geodesics for $g$.
We say that $\KK$ is \emph{homologically visible} when $C_*(\KK)$ is non-trivial (the analogous notion was introduced for isolated compact sets of closed geodesics in $\Gamma(\KK)$ in Section~\ref{ss:hom_visible}). 
The length filtration mentioned above implies that there is an isomorphism
\begin{align*}
C_*(\KK)
\cong
\varinjlim_\ell H_*(\overline\KK^{<\ell}\cup\overline\KK_\rho,\overline\KK_\rho),
\end{align*}
where the direct limit is for $\ell\to\infty$.
This, together with the above properties (i) and (ii), implies that any homologically visible primitive relative flat knot type contains a closed geodesic. If $\Gamma_g(\KK)$ is discrete, we also have the following version of the classical Morse inequalities.

\begin{Prop}\label{p:Morse_inequality}
For each primitive relative flat knot type $\KK$, if the space of closed geodesics $\Gamma_g(\KK)$ is discrete, then
\begin{align*}
 \rank(C_d(\KK)) \leq \sum_{\gamma\in\Gamma(\KK)} \rank(C_d(\gamma)),\qquad
 \forall d\geq 1.
\end{align*}
\end{Prop}

\begin{proof}
Assume that $\Gamma(\KK)$ is discrete.
For each $[a,b)\subset\R$ such that $(a,b)\cap\sigma(\KK)=\ell$, property (ii) above implies
\begin{align}
\label{e:rank_level_homology}
\rank\big(H_d(\overline\KK^{<b}\cup\overline\KK_\rho,\overline\KK^{<a}\cup\overline\KK_\rho)\big)
=
 \sum_{\gamma}
 \rank(C_d(\gamma)),
\end{align}
where the sum on the right-hand side ranges over all closed geodesics $\gamma\in\Gamma(\KK)$ of length $L(\gamma)=\ell$.
We recall that the relative homology is sub-additive, meaning that $\rank(H_d(A,C))\leq\rank(H_d(A,B))+\rank(H_d(B,C))$ for all spaces $C\subseteq B\subseteq A$, see \cite[Section~5]{Milnor:1963aa}. This, together with~\eqref{e:rank_level_homology}, implies
\begin{align*}
\label{e:rank_level_homology}
\rank\big(H_d(\overline\KK^{<\ell}\cup\overline\KK_\rho,\overline\KK_\rho)\big)
=
 \sum_{\gamma}
 \rank(C_d(\gamma)),
 \qquad\forall \ell>0,
\end{align*}
where the sum on the right-hand side ranges over all closed geodesics $\gamma\in\Gamma(\KK)$ of length $L(\gamma)<\ell$. By taking the direct limit for $\ell\to\infty$, we obtain the desired inequality.
\end{proof}

A $C^\infty$-generic Riemannian metric is bumpy \cite{Anosov:1982aa}, meaning that all  closed geodesics are non-degenerate. In particular, for every such metric, the whole space of closed geodesics is a discrete subspace of $\Omega$, and we obtain the following corollary.

\begin{Cor}
If the Riemannian metric $g$ is bumpy, then each primitive relative flat knot type $\KK$ contains at least $\rank(C_*(\KK))$ closed geodesics.
\end{Cor}

\begin{proof}
The assertion is a direct consequence of the Morse inequality of Proposition~\ref{p:Morse_inequality} and Lemma~\ref{l:non_deg_visible}.
\end{proof}

The following lemma is a special case of Morse lacunary principle in the setting of primitive relative flat knot types.

\begin{Lemma}
\label{l:local_homology_parity}
Let $\KK$ be a primitive relative flat knot type. If $\Gamma(\KK)$ is non-empty, contains only non-degenerate closed geodesics, and their Morse indices have the same parity, then $\KK$ is homologically essential, and
\[
C_*(\KK) \cong \bigoplus_{\gamma\in\Gamma(\KK)} C_*(\gamma).
\] 
\end{Lemma}

\begin{proof}
Since all closed geodesics in $\Gamma(\KK)$ are non-degenerate, in particular $\Gamma(\KK)$ and $\sigma(\KK)$ are discrete. By assumption, there exist $q\in\{0,1\}$ such that every closed geodesic $\gamma\in\Gamma(\KK)$ has Morse index of the same parity as $q$. Therefore the local homomology of every such $\gamma$ is trivial in all degrees of the same parity as $q+1$. This, together with property (ii) above, implies that, for each $a<b$ such that $[a,b)\cap\sigma(\KK)$ contains only one element, the relative homology group 
\[H_d(\overline\KK^{<a}\cup\overline\KK_\rho,\overline\KK^{<b}\cup\overline\KK_\rho)\] 
vanishes in all degrees $d$ of the same parity as $q+1$. Therefore, the inclusion~\eqref{e:inclusion_sublevel_sets_within_K} induces a short exact sequence
\begin{align*}
0
\toup
H_*(\overline\KK^{<a}\cup\overline\KK_\rho,\overline\KK_\rho)
\toup
H_*(\overline\KK^{<b}\cup\overline\KK_\rho,\overline\KK_\rho)
\toup
H_*(\overline\KK^{<b}\cup\overline\KK_\rho,\overline\KK^{<a}\cup\overline\KK_\rho)
\toup
0.
\end{align*}
This implies that, for each $\ell>0$, 
\begin{align*}
 H_*(\overline\KK^{<\ell}\cup\overline\KK_\rho,\overline\KK_\rho)
 \cong
 \bigoplus_{\gamma\in\Gamma(\KK)\cap L^{-1}[0,\ell)}
 C_*(\gamma).
\end{align*}
After taking a direct limit for $\ell\to\infty$, we infer
\begin{align*}
 C_*(\KK)
 \cong
 \bigoplus_{\gamma\in\Gamma(\KK)}
 C_*(\gamma).
\end{align*}
Since $\Gamma(\KK)$ is assumed to contain at least one closed geodesic, which is homologically visible being non-degenerate (Lemma~\ref{l:non_deg_visible}), we conclude that the local homology $C_*(\KK)$ is non-trivial.
\end{proof}

\section{$C^0$-stability of flat links of closed geodesics}\label{s:flat_links}

\subsection{The simplest case: flat knots of closed geodesics}

In this subsection, we shall prove the much simpler version of Theorem~\ref{mt:links} for the special case of flat links with only one component, that is, flat knots. Actually, the result for flat knots, Theorem~\ref{t:knots} below, has weaker assumptions: it only requires homotopically visible spectral values, as opposed to homologically visible ones, and does not even need the involved closed geodesics to be isolated.

Let $(M,g)$ be a closed  Riemannian surface, and $\KK$ a primitive flat knot type.
We define the \emph{visible $\KK$-spectrum} 
\[\sigmav(\KK)=\sigmavv{g}(\KK)\] 
to be the set of positive real numbers $\ell>0$ such that, for any sufficiently small neighborhood $[\ell_-,\ell_+]$ of $\ell$ and for some (and thus for all) $\rho\in(0,\rho_g]$, the inclusion 
\begin{align*}
 \overline\KK{}^{<\ell_-}\cup\overline\KK_{\rho}
 \hookrightarrow
 \overline\KK{}^{<\ell_+}\cup\overline\KK_{\rho}
\end{align*}
is not a homotopy equivalence. It will follow from Lemmas~\ref{l:non_deg_visible} and~\ref{l:local_homology_parity} that the length $L(\gamma)$ of any closed geodesic $\gamma\in\Gamma(\KK)$ that is non-degenerate, or more generally homologically visible, belongs to $\sigmav(\KK)$. Conversely, we have the following lemma.

\begin{Lemma}
$\sigmav(\KK)\subseteq\sigma(\KK)$.
\end{Lemma}

\begin{proof}
Since $\sigma(\KK)$ is a closed subset of $\R$, for any $\ell\in(0,\infty)\setminus\sigma(\KK)$ we can find $\ell_0<\ell_1<\ell<\ell_2$ such that $[\ell_0,\ell_2]\cap\sigma(\KK)=\varnothing$. We define the continuous function 
\begin{align*}
\tau:\overline\KK{}^{< \ell_2}\cup\overline\KK_\rho\to[0,\infty],
\qquad
\tau(\gamma) = \sup\big\{ t\in[0,\tau_\rho(\gamma))\ \big|\  L(\phi^t(\gamma))\geq \ell_0\big\},
\end{align*}
where $\tau_\rho$ is the exit time function of Lemma~\ref{l:Angenent_exit_time}.
Notice that, actually, $\tau$ is everywhere finite. Indeed, if $\tau(\gamma)=\infty$, then there would exist a sequence $t_n\to\infty$ with $\phi^{t_n}(\gamma)$  converging to a closed geodesic $\zeta\in\Gamma(\KK)$ of length $L(\zeta)\in[\ell_0,\ell_2]$, contradicting the fact that $[\ell_0,\ell_2]\cap\sigma(\KK)=\varnothing$. The map
\begin{align*}
 \overline\KK{}^{<\ell_2}\cup\overline\KK_{\rho}
 \to
 \overline\KK{}^{<\ell_1}\cup\overline\KK_{\rho},
 \qquad
 \gamma\mapsto\phi^{\tau(\gamma)}(\gamma).
\end{align*}
is a homotopy inverse of the inclusion $\overline\KK{}^{<\ell_1}\cup\overline\KK_{\rho}
 \hookrightarrow
 \overline\KK{}^{<\ell_2}\cup\overline\KK_{\rho}$.
\end{proof}

We now prove the anticipated $C^0$-stability of visible spectral values of primitive flat knot types.

\begin{Thm}\label{t:knots}
Let $(M,g)$ be a closed Riemannian surface, $\KK$ a primitive flat knot type, and $\ell\in\sigmavv{g}(\KK)$ a visible spectral value. For each $\epsilon>0$, any Riemannian metric $h$ sufficiently $C^0$-close to $g$ has a visible spectral value in $(\ell-\epsilon,\ell+\epsilon)$, i.e.
\[\sigmavv{h}(\KK)\cap(\ell-\epsilon,\ell+\epsilon)\neq\varnothing.\] 
\end{Thm}

\begin{proof}
In order to simplify the notation, for each $b>0$ and $\rho>0$ we denote
\begin{align*}
 \GG(b,\rho):=
 \overline\KK{}_{g}^{<b}\cup\overline\KK_{g,\rho}.
\end{align*}
Let $\epsilon>0$ be small enough so that, for any neighborhood $[\ell_-,\ell_+]\subset(\ell-\epsilon,\ell+\epsilon)$ of $\ell$ and for any $\rho\in(0,\rho_g]$, the inclusion $\GG(\ell_-,\rho)\hookrightarrow\GG(\ell_+,\rho)$ is not a homotopy equivalence.
Since the $\KK$-spectrum $\sigma_g(\KK)$ is closed and has measure zero, there exist values
\[\ell-\epsilon<r_1<r_2<r_3<\ell<s_1<s_2<s_3<\ell+\epsilon\] 
such that \[\sigma_g(\KK)\cap\big([r_1,r_3]
\cup[s_1,s_3]\big)=\varnothing.\] 
We fix $\delta>1$ close enough to $1$ so that $r_i\delta\leq r_{i+1}$ and $s_i\delta\leq s_{i+1}$ for all $i=1,2$, and $r_3\delta\leq s_1$. Let $h$ be a Riemannian metric on $M$ such that
\begin{align}\label{e:g-h}
 \delta^{-1}\|\cdot\|_h
 \leq
 \|\cdot\|_{g}
 \leq
 \delta\|\cdot\|_h,
 \qquad
 \delta^{-1}\mu_h
 \leq
 \mu_{g}
 \leq
 \delta\mu_h,
\end{align}
where $\mu_g$ and $\mu_h$ are the Riemannian densities on $M$ associated with $g$ and $h$ respectively.
For each $b>0$ and $\rho>0$ we denote
\begin{align*}
 \HH(b,\rho):=
 \overline\KK{}_{h}^{<b}\cup\overline\KK_{h,\rho}.
\end{align*}
Notice that, by~\eqref{e:g-h}, we have
\begin{align*}
\GG(b,\rho)\subseteq\HH(b\delta,\rho\delta)\subseteq\GG(b\delta^2,\rho\delta^2). 
\end{align*}
We fix $0<\rho_1<\rho_2<\rho_3<\rho_4<\rho_5\leq\min\{\rho_g,\rho_h\}$ such that $\rho_i\delta\leq \rho_{i+1}$ for all $i=1,2,3,4$. We obtain a diagram of inclusions
\[
\begin{tikzcd}[row sep=large]
    \GG(r_1,\rho_1)
    \arrow[rr, hook] 
    \arrow[drr, hook, "i_1"', "\simeq"] 
    && 
    \HH(r_2,\rho_2)
    \arrow[d, hook] 
    \arrow[rrrr, hook, "j"]
    &&&&
    \HH(s_2,\rho_2)
    \arrow[d, hook, "i_4"', "\simeq"]
    \\
    &&
	\GG(r_3,\rho_3)
	\arrow[rr, hook,"\not\simeq","i_2"']&&
    \GG(s_1,\rho_3) 
    \arrow[rr, hook] 
    \arrow[drr, hook, "i_3"', "\simeq"] 
    &&
    \HH(s_2,\rho_4)
    \arrow[d, hook] 
    \\
    &&&&&&
	\GG(s_3,\rho_5)
\end{tikzcd}
\]
Since  $i_1$, $i_3$, and $i_4$ are homotopy  equivalences, whereas $i_2$ is not a homotopy equivalence, we infer that $j$ 
is not a homotopy equivalence neither, and therefore 
\[
\sigmavv{h}(\Sigma)\cap[r_2,s_2)\neq\varnothing.
\qedhere
\]
\end{proof}

\subsection{Intertwining curve shortening flow trajectories}

In order to prove the $C^0$-stability for flat links of closed geodesics (Theorem~\ref{mt:links}), we first need a refinement of Theorem~\ref{t:knots} that not only provides the $C^0$-stability of a closed geodesic of given primitive flat knot type, but also connects neighborhoods of  corresponding closed geodesics of the old and new metrics by means of curve shortening flow lines.

Let $\KK$ be a primitive flat knot type. As in the proof of Theorem~\ref{t:knots}, in order to simplify the notation we set
\begin{align*}
\GG(b):=
 \overline\KK{}_{g}^{<b},
 \qquad
 \GG(b,\rho):=
 \GG(b)\cup\overline\KK_{g,\rho}.
\end{align*}
Let $\gamma\in\Gamma_g(\KK)$ be a homologically visible closed geodesic of length $\ell:=L_g(\gamma)$. By Lemma~\ref{l:GM_nbhds}, we can find an arbitrarily $C^2$-small Gromoll-Meyer neighborhood $\UU$ of $\gamma$, and an arbitrarily $C^2$-small Gromoll-Meyer neighborhood $\VV$ of the compact set of closed geodesics $\Gamma_g(\KK)\cap L_g^{-1}(\ell)\setminus\{\gamma\}$. We require $\UU$ and $\VV$ to be small enough so that $\UU\cap\Gamma_g(\KK)=\{\gamma\}$ and  $\UU\cap\VV=\varnothing$. Let $\epsilon>0$ be sufficiently small so that
\begin{align}
\label{e:GM_property_U'_V'}
 (\Phi_g(\UU)\setminus\UU) \cup (\Phi_g(\VV)\setminus\VV)\subset\GG(\ell-\epsilon).
\end{align}
We denote by $\Sigma$ a relative cycle representing a non-zero element of the local homology group $C_*(\gamma)=H_*(\UU\cup \GG(\ell-\epsilon),\GG(\ell-\epsilon))$. By an abuse of terminology, we will occasionally forget the relative cycle structure of $\Sigma$, and simply treat it as a compact subset of $\UU\cup \GG(\ell-\epsilon)$. We will refer to such a $\Sigma$ as to a \emph{Gromoll-Meyer relative cycle}.

\begin{Lemma}\label{l:deformation_lemma}
For each sufficiently small neighborhood $[\ell_-,\ell_+]$ of $\ell$, for each Riemannian metric $h$ sufficiently $C^0$-close to $g$,  for each $C^3$-open neighborhood $\WW$ of $\Gamma_h(\KK)\cap L_h^{-1}[\ell_-,\ell_+]$, and for each $\rho\in(0,\rho_h]$ small enough, the following points hold. 
\begin{itemize}
\setlength{\itemsep}{5pt}

\item[$(i)$]  
We consider the following modified curve shortening flow of $h$, which stops the orbits once they enter $\overline\KK_{h,\rho}$$:$
\begin{align*}
\psi_h^t(\zeta):=\phi_h^{\max\{t,\tau_{h,\rho}(\zeta)\}}(\zeta),\qquad\forall\zeta\in\KK,
\end{align*}
where $\tau_{h,\rho}$ is the exit-time function of Lemma~\ref{l:Angenent_exit_time}.
For each $t\geq0$, there exists $\tau_1(t)\geq0$ such that 
\[\psi_h^{[0,\tau_1(t)]}(\zeta)\, \cap\, \big(\WW\cup\HH(\ell_-,\rho)\big)\neq\varnothing,
\qquad
\forall \zeta\in\psi_h^t(\Sigma).\]
In particular
\[
\psi_h^{\tau_1(0)}(\Sigma)
\subset
\Phi_h(\WW)\cup \HH(\ell_-,\rho).
\]
where $\HH(\ell_-,\rho):=
\overline\KK{}_{h}^{<\ell_-}\cup\overline\KK_{h,\rho}$.

\item[$(ii)$] For each $t\geq\tau_1(0)$, there exists $\tau_2(t)\geq0$ and, for each $s\geq\tau_2(t)$, there exists $\zeta\in\Sigma$ such that $t\leq\tau_{h,\rho}(\zeta)$ and
\begin{align*}
\phi_h^{[0,t]}(\zeta) \cap \WW\neq\varnothing,
\qquad
\phi_h^{t}(\zeta)\not\in\HH(\ell_-,\rho)
\qquad
\phi_g^s\circ\phi_h^t(\zeta)\in\UU.
\end{align*}
\end{itemize}
\end{Lemma}

\begin{proof}
While the Gromoll-Meyer neighborhood $\UU$ is not open, it contains a $C^3$-open neighborhood $\UU'\subset\UU$ of $\gamma$. Analogously, $\VV$ contains a $C^3$-open neighborhood $\VV'\subset\VV$ of $\Gamma_g(\KK)\setminus\{\gamma\}$.
For a sufficiently small neighborhood $[a_0,a_2]\subset(\ell-\epsilon,\infty)$ of $\ell$, all the closed geodesics in $\Gamma_g(\KK)\setminus\{\gamma\}$ of length in $[a_0,a_2]$ are contained in $\VV'$, i.e.
\begin{align*}
 \Gamma_g(\KK)\cap L_g^{-1}[a_0,a_2]\setminus\{\gamma\}\subset\VV'.
\end{align*}
We require $a_2\not\in\sigma_g(\KK)$, which is possible since the $\KK$-spectrum $\sigma_g(\KK)$ is closed and has measure zero. Therefore there exists $a_1\in(\ell,a_2)$ such that 
\[
[a_1,a_2]\subset(\ell,a_2]\setminus\sigma_g(\KK).
\] 
Up to replacing $\UU$ and $\VV$ with $\UU\cap\GG(a_1)$ and $\VV\cap\GG(a_1)$ respectively, we can assume without loss of generality that
\begin{align*}
\UU\cup\VV\subset\GG(a_1). 
\end{align*}
By the excision property of singular homology, the inclusions
\begin{align*}
&i_1:\big(\UU\cup\GG(a_0) , \GG(a_0)\big)
\hookrightarrow 
\big(\UU\cup\VV\cup\GG(a_0) , \GG(a_0)\big),\\
&i_2:\big(\VV\cup\GG(a_0) , \GG(a_0)\big)
\hookrightarrow 
\big(\UU\cup\VV\cup\GG(a_0) , \GG(a_0)\big)
\end{align*}
induce an isomorphism
\begin{equation}
\label{e:GM_splitting}
\begin{tikzcd}
H_*(\UU\cup\GG(a_0) , \GG(a_0))
\oplus 
H_*(\VV\cup\GG(a_0) , \GG(a_0))
\arrow[d,"\cong","i_{1*}\oplus i_{2*}"']\\
H_*(\UU\cup\VV\cup\GG(a_0) , \GG(a_0))
\end{tikzcd} 
\end{equation}

We fix a constant $\delta>1$ close enough to $1$ so that
$\delta^2a_0<\ell$ and $\delta^3a_1<a_2$.
Let $h$ be any Riemannian metric on $M$ that is sufficiently $C^0$-close to $g$ so that
\begin{align*}
 \delta^{-1}\|\cdot\|_h
 \leq
 \|\cdot\|_{g}
 \leq
 \delta\|\cdot\|_h,
 \qquad
 \delta^{-1}\mu_h
 \leq
 \mu_{g}
 \leq
 \delta\mu_h,
\end{align*}
where $\mu_h$ and $\mu_g$ are the Riemannian densities on $M$ associated with $h$ and $g$ respectively.
For the Riemannian metric $h$, we introduce the notation
\begin{align*}
\HH(b):=\overline\KK{}_{h}^{<b},
\qquad
 \HH(b,\rho):=
\HH(b)\cup\overline\KK_{h,\rho}.
\end{align*}
Let $\WW\subset\HH(\delta^2a_1)\subset\GG(a_2)$ be a $C^3$-open neighborhood of $\Gamma_h(\KK)\cap L_h^{-1}[\delta a_0,\delta a_1]$. We fix $0<\rho_1<\rho_2<\rho_3$ such that $\delta\rho_1<\rho_2$, $\delta\rho_2<\rho_3$, $\rho_2<\rho_h$, and $\rho_3<\rho_g$. The constant $\rho_2$ will be the $\rho$ in the statement of the lemma, and therefore we set
\begin{align*}
\psi_h^t(\zeta):=\phi_h^{\max\{t,\tau_{h,\rho_2}(\zeta)\}}(\zeta),\qquad\forall\zeta\in\overline\KK.
\end{align*}

Since 
$\GG(a_1,\rho_1)
\subset
\HH(\delta a_1,\rho_2)$, the arrival-time function
\begin{gather*}
  s_1 :\GG(a_1,\rho_1)\to[0,\infty),
 \qquad
 s_1(\zeta) =\inf\big\{t\geq0\ \big|\ \psi_h^t(\zeta)\in \WW\cup\HH(\delta a_0,\rho_2) \big\}
\end{gather*}
is everywhere finite. Moreover, since $\psi_h^t$ preserves the subset $\Phi_h(\WW)\cup\HH(\delta a_0,\rho_2)$, we have
\begin{align*}
\psi_h^t(\zeta)\in\Phi_h(\WW)\cup\HH(\delta a_0,\rho_2),
\qquad
\forall\zeta\in\UU\cup\VV\cup\GG(a_0,\rho_1),\ t> s_1(\zeta).
\end{align*}
Since the arrival set $\WW\cup\HH(\delta a_0,\rho_2)$ is open and $\psi_h^t$ is continuous, $s_1$ is upper semi-continuous. Our loop space $\Omega$ is Hausdorff and metrizable, and in particular admits a partition of unity subordinated to any given open cover. Since $s_1$ is upper semicontinuous, by means of a suitable partition of unity we can construct a continuous function $\sigma_1:\GG(a_1,\rho_1)\to[0,\infty)$ such that $\sigma_1(\zeta)> s_1(\zeta)$ for all $\zeta\in \GG(a_1,\rho_1)$. Since $\UU\cup\VV\cup\GG(a_0,\rho_1)\subset\GG(a_1,\rho_1)$, we can build a continuous map
\begin{align*}
 \nu_1: \UU\cup\VV\cup\GG(a_0,\rho_1)\to\Phi_h(\WW)\cup\HH(\delta a_0,\rho_2),
 \qquad
 \nu_1(\zeta)=\psi_h^{\sigma_1(\zeta)}(\zeta).
\end{align*}

We now introduce a modified curve shortening flow for the Riemannian metric $g$, which stops the orbits once they enter $\overline\KK_{g,\rho_1}$$:$
\begin{align*}
\psi_g^t(\zeta):=\phi_g^{\max\{t,\tau_{g,\rho_1}(\zeta)\}}(\zeta),\qquad\forall\zeta\in\overline\KK.
\end{align*}
We consider the open sets $\UU'\subset\UU$ and $\VV'\subset\VV$ introduced at the beginning of the proof.
Since their union $\UU'\cup\VV'$ contains $\Gamma_g(\KK)\cap L_g^{-1}[a_0,a_2]$, in particular the arrival-time function
\begin{gather*}
  s_2 :\GG(a_2,\rho_3)\to[0,\infty),
 \qquad
 s_2(\zeta) =\inf\big\{t\geq0\ \big|\ \psi_g^t(\zeta)\in\UU'\cup\VV'\cup\GG(a_0,\rho_1) \big\}
\end{gather*}
is everywhere finite. By~\eqref{e:GM_property_U'_V'} and $\GG(\ell-\epsilon)\subset\GG(a_0,\rho_1)$, the semi-flow $\psi_g^t$ preserves $\UU\cup\VV\cup\GG(a_0,\rho_1)$, and therefore
\begin{align*}
 \psi_g^t(\zeta)\in\UU\cup\VV\cup\GG(a_0,\rho_1),
 \qquad
 \forall\zeta\in\GG(a_2,\rho_3),\ t>s_2(\zeta).
\end{align*}
Once again, since the arrival set $\UU'\cup\VV'\cup\GG(a_0,\rho_1)$ is open, the arrival-time function $s_2$ is upper semi-continuous, and by means of a partition of unity we construct a continuous function  $\sigma_2:\GG(a_2,\rho_3)\to[0,\infty)$ such that $\sigma_2(\zeta)> s_2(\zeta)$ for all $\zeta\in\GG(a_2,\rho_3)$. Therefore, we obtain a continuous map
\begin{align*}
 \nu_2:\GG(a_2,\rho_3)\to \UU\cup\VV\cup\GG(a_0,\rho_1),
 \qquad
 \nu_2(\zeta)=\psi_g^{\sigma_2(\zeta)}(\zeta).
\end{align*}

Since $\GG(a_0,\rho_1)\subset\HH(\delta a_0,\rho_2)\subset\GG(\ell,\rho_3)$ and $\Phi_h(\WW)\subset\HH(\delta^2 a_1)\subset\GG(a_2)$, overall we obtain a diagram
\[
\begin{tikzcd}
    \big(\UU\cup\VV\cup\GG(a_0,\rho_1) , \GG(a_0,\rho_1)\big)
    \arrow[dd, "\nu_1"] \\
    & 
    \big(\GG(a_2,\rho_3) ,  \GG(\delta^2 a_0,\rho_3)\big)
    \arrow[ul, "\nu_2"']
    \\
    \big(\Phi_h(\WW)\cup\HH(\delta a_0,\rho_2),\HH(\delta a_0,\rho_2)\big)
    \arrow[ur, hook, "i"]
\end{tikzcd}
\]
where $i$ is an inclusion. The composition $\nu_2\circ i$ is a homotopy inverse of $\nu_1$, and in particular $\nu_2\circ\nu_1$ induces the identity isomorphism on the relative homology group
$H_*\big(\UU\cup\VV\cup\GG(a_0,\rho_1) , \GG(a_0,\rho_1)\big)$.

Consider the Gromoll-Meyer relative cycle $\Sigma$ introduced before the statement, and fix a value $t\geq0$. We claim that there exists $\tau_1(t)\geq0$ such that
\[\psi_h^{[0,\tau_1(t)]}(\zeta)\, \cap\, \big(\WW\cup\HH(\delta a_0,\rho_2)\big)\neq\varnothing,
\qquad
\forall \zeta\in\psi_h^t(\Sigma).\]
Indeed, assume by contradiction that such a $\tau_1(t)$ does not exist. Therefore there exists a sequence $\zeta_n\in\psi_h^t(\Sigma)$ such that 
\[\psi_h^{[0,n]}(\zeta_n)\, \cap\, \big(\WW\cup\HH(\delta a_0,\rho_2)\big)=\varnothing.\]
Since $\psi_h^t(\Sigma)$ is compact, we can extract a subsequence of $\zeta_n$ converging to some $\zeta\in\psi_h^t(\Sigma)$, and we have $\psi_h^s(\zeta)=\phi_h^s(\zeta)\not\in\WW\cup\HH(\delta a_0,\rho_2)$ for all $s\geq0$. This implies that, for some $s_n\to\infty$, the sequence $\phi_h^{s_n}(\zeta)$ converges to a closed geodesic in $\Gamma_h(\KK)\cap L_h^{-1}[\delta a_0,\delta a_1]$. However, this latter set is contained in $\WW$, which gives a contradiction. This proves point (i).

As for point (ii), for a fixed value $t\geq\tau_1(0)$, we set
\begin{align*}
 \tau_2(t):=\max_{\zeta\in\psi_h^t(\Sigma)} \sigma_2(\zeta).
\end{align*}
We set
$\Sigma':=\big\{\zeta\in\Sigma\ \big|\ \psi_h^t(\zeta)\not\in \HH(\delta a_0,\rho_2) \big\}$.
Notice that 
\[
\phi_h^{[0,t]}(\zeta)\cap \WW
=\psi_h^{[0,t]}(\zeta)\cap \WW\neq\varnothing,
\qquad\forall\zeta\in\Sigma'.
\]
Moreover, since $\psi_h^{t}(\Sigma\setminus\Sigma') \subset \HH(\delta a_0,\rho_2)$, we have
\begin{align}
\label{l:pushed_below_lemma}
 \psi_g^{s}\circ \psi_h^{t}(\Sigma\setminus\Sigma')\subset\GG(a_0,\rho_1),\qquad\forall s\geq\tau_2(t).
\end{align}
We are left to show that, for all $s\geq\tau_2(t)$, there exists $\zeta\in\Sigma'$ such that \[\phi_g^s\circ\phi_h^t(\zeta)\in\UU.\]
Assume by contradiction that this does not hold, so that, in particular, for some $s\geq\tau_2(t)$ we have
$\psi_g^s\circ\psi_h^t(\Sigma')\cap\UU\setminus\GG(a_0,\rho_1)=\varnothing$. This and~\eqref{l:pushed_below_lemma} imply that
\begin{align*}
\psi_g^s\circ\psi_h^t(\Sigma)
\subset
\VV\cup\GG(a_0,\rho_1)
\end{align*}
However,
\begin{align*}
(\nu_2\circ\nu_1)_*[\Sigma]=[\psi_g^s\circ\psi_h^t(\Sigma)].
\end{align*}
This, together with the splitting~\eqref{e:GM_splitting} and the excision
\[
H_*(\UU\cup\VV\cup\GG(a_0) , \GG(a_0))
\ttoup^{\cong}
H_*(\UU\cup\VV\cup\GG(a_0,\rho_1) , \GG(a_0,\rho_1)),
\]
implies that $(\nu_2\circ\nu_1)_*[\Sigma]$ belongs to the direct summand $H_*(\VV\cup\GG(a_0,\rho_1) , \GG(a_0,\rho_1))$ of the relative homology group $H_*(\UU\cup\VV\cup\GG(a_0,\rho_1) , \GG(a_0,\rho_1))$. This contradicts the fact that $(\nu_2\circ\nu_1)_*$ is the identity in relative homology.
\end{proof}

\subsection{Primitive flat link types}
We now consider a finite collection of  homologically visible closed geodesics $\gamma_i\in\Gamma_g(\KK_i)\cap L_g^{-1}(\ell_i)$, for $i=1,...,n$, where the $\KK_i$'s are flat knot types. Let $\UU_i$ be a Gromoll-Meyer neighborhood of $\gamma_i$ such that $\UU_i\cap\Gamma_g(\KK_i)=\{\gamma_i\}$. We apply Lemma~\ref{l:deformation_lemma} simultaneously to all these closed geodesics, and obtain the following statement, which is the last ingredient for the proof of Theorem~\ref{mt:links}.

\begin{Lemma}\label{l:intertwining}
For each $\epsilon>0$, for each Riemannian metric $h$ sufficiently $C^0$-close to $g$, and for each collection of $C^2$-open neighborhoods $\ZZ_i$ of $\Gamma_h(\KK_i)\cap L_h^{-1}[\ell_i-\epsilon,\ell_i+\epsilon]$, there exist $t_1,t_2,t_3\in[0,\infty)$ and, for all $i\in\{1,...,n\}$, an element $\zeta_i\in\UU_i$ such that 
\[\phi_h^{t_1}(\zeta_i)\in\ZZ_i,
\qquad
\phi_g^{t_3}\circ\phi_h^{t_2+t_1}(\zeta_i)\in\UU_i.\]
\end{Lemma}

\begin{proof}
We fix Gromoll-Meyer relative cycles $\Sigma_i\subset\UU_i$ for each $\gamma_i$.
Notice that it is enough to prove the lemma for small values of $\epsilon>0$, as the statement would then hold for larger values of $\epsilon$ as well. We require $\epsilon$ to be small enough so that we can apply Lemma~\ref{l:deformation_lemma} simultaneously to all closed geodesics $\gamma_i$ with the neighborhood $[\ell_i-\epsilon,\ell_i+\epsilon]$ of their length and their Gromoll-Meyer relative cycles $\Sigma_i\subset\UU_i$. We consider the compact sets of closed geodesics
\[\Gamma_i:=\Gamma_h(\KK_i)\cap L_h^{-1}[\ell_i-\epsilon,\ell_i+\epsilon],\qquad i=1,...,n,\]
and $C^2$-open neighborhoods $\ZZ_i\subset L_h^{-1}(\ell_i-2\epsilon,\ell_i+2\epsilon)$ of $\Gamma_i$.
By Lemma~\ref{l:encapsulated_nbhds}, there exist $\delta>0$ and, for each $i=1,...,n$, a $C^3$-open neighborhood $\WW_i\subset\ZZ_i$ of $\Gamma_i$ such that, whenever $\zeta\in\WW_i$ and $\phi_h^t(\zeta)\not\in\ZZ_i$, we have $L(\zeta)-L(\phi_h^t(\zeta))\geq\delta$. Notice that  $N:=\lfloor 4\epsilon/\delta \rfloor$ is an upper bound for the number of times that any orbit $\phi_h^t(\zeta)$ can go from outside $\ZZ_i$ to inside $\WW_i$, i.e.
\begin{align*}
N
\geq
\max_{i=1,...,n}\
\max_{\zeta\in\KK_i}\ \#\left\{k\geq0\ \left|  
\begin{array}{l@{}}
\exists\ 0\leq a_1< b_1< ...< a_k< b_k\mbox{ and } \zeta\in\KK_i\\
\mbox{such that}\\ 
\phi_h^{a_j}(\zeta)\not\in\ZZ_i,\ \phi_h^{b_j}(\zeta)\in\WW_i,\ \forall j=1,...,k
\end{array}
\right.\right\}.
\end{align*}
Lemma~\ref{l:deformation_lemma} provides $\rho\in(0,\rho_h]$ and two functions $\tau_{i,1}$ and $\tau_{i,2}$ satisfying the properties stated for the data associated to the closed geodesic $\gamma_i$. We set
\begin{align*}
 \tau_1(t):=\max_{i=1,...,n} \tau_{i,1}(t),\qquad
 \tau_2(t):=\max_{i=1,...,n} \tau_{i,2}(t),
\end{align*}
so that the functions $\tau_1$ and $\tau_2$ satisfy the properties stated in Lemma~\ref{l:deformation_lemma} with respect to the data associated to any of the closed geodesics $\gamma_i$. We define a sequence of real numbers $t_k$ and $s_k$, for $k\geq0$, by 
\[
t_0:=0,
\qquad
t_{k+1}:=t_k+\tau_1(t_k),
\qquad
s_k:=\tau_2(t_k).
\]
By Lemma~\ref{l:deformation_lemma}(ii), there exist $\zeta_{i,k}\in\Sigma_i$ such that 
\begin{align*}
 \phi_h^{[0,t_k]}(\zeta_{i,k})\cap \WW_i \neq \varnothing,
 \qquad
 \phi_h^{t_k}(\zeta_{i,k})\not\in\HH_i(\ell_i-\epsilon,\rho),
 \qquad
 \phi_g^{s_k}\circ\phi_h^{t_k}(\zeta_{i,k})\in\UU_i,
\end{align*}
where $\HH_i(\ell_i-\epsilon,\rho):=\overline\KK{}_{i,h}^{<\ell_i-\epsilon}\cup\overline\KK_{i,h,\rho}$. We set
\[
J(i,k):=\Big\{ j\in\{1,...,k-1\}\ \Big|\ \phi_h^{t_j}(\zeta_{i,k})\not\in\ZZ_i \Big\}.
\]
By Lemma~\ref{l:deformation_lemma}(i), for each $j\in J(i,k)$ we have 
\[\phi_h^{[t_j,t_{j+1}]}(\zeta_{i,k})\cap\WW_i\neq\varnothing.\]
Namely, the loop $\phi_h^{t}(\zeta_{i,k})$ is outside $\ZZ_i$ for $t=t_j$, but enters $\WW_i$ for some $t\in[t_j,t_{j+1}]$. This implies the cardinality bound $\# J(i,k) \leq N$.
Therefore, the set
\[
J(k):=\bigcup_{i=1,...,n} J(i,k)
\]
has cardinality $\#J(k)\leq nN$. For any $k\geq nN+2$, the set $\{1,...,k-1\}\setminus J(k)$ is non-empty, and for every $j\in\{1,...,k-1\}\setminus J(k)$ we have
\[
 \phi_h^{t_{j}}(\zeta_{i,k})\in\ZZ_i,
 \qquad
 \phi_g^{s_k}\circ\phi_h^{t_{k}}(\zeta_{i,k})\in\UU_i,
 \qquad\forall i=1,...,k.
\qedhere
\]
\end{proof}

We can now provide the proof of Theorem~\ref{mt:links}. Actually, we will prove the following slightly stronger statement, which relaxes the non-degeneracy condition of Definition~\ref{d:stable_link}, and replaces it with the homological visibility.

\begin{Thm}\label{t:links}
Let $(M,g)$ be a closed Riemannian surface, $\LL$ a primitive flat link type, and $\bm\gamma=(\gamma_1,...,\gamma_n)\in\LL$ a flat link of homologically visible closed geodesics such that, for each $i\neq j$, the components $\gamma_i,\gamma_j$ have distinct flat knot types or distinct lengths $L_g(\gamma_i)\neq L_g(\gamma_j)$. For each $\epsilon>0$, any Riemannian metric $h$ sufficiently $C^0$-close to $g$ has a flat link of closed geodesics $\bm\zeta\in\LL$ and such that $\| L_h(\bm\zeta)- L_g(\bm\gamma)\|<\epsilon$.
\end{Thm}

\begin{proof}
Let $\KK_i$ be the flat knot type of the component $\gamma_i$, and $\ell_i:=L_g(\gamma_i)$ its length. Let $\epsilon>0$ be small enough so that, for all $i\neq j$, either $\ell_i=\ell_j$ or $|\ell_i-\ell_j|>2\epsilon$. Let $h$ be a Riemannian metric that is sufficiently $C^0$-close to $g$ so that Lemma~\ref{l:intertwining} holds. For each $i$, we have a non-empty compact set of closed geodesics 
\[\Gamma_i:=\Gamma_h(\KK_i)\cap L_h^{-1}[\ell_i-\epsilon,\ell_i+\epsilon].\] 
We fix Gromoll-Meyer neighborhoods $\UU_i\subset\KK_i$ of the components $\gamma_i$ such that $\UU_i\cap\Gamma_g(\KK_i)=\{\gamma_i\}$, and $C^2$-open neighborhoods $\ZZ_{i}\subset\KK_i$ of $\Gamma_i$.
Since $\Gamma_i$ is compact, $\ZZ_i$ has only finitely many connected components $\ZZ_{i,1},...\ZZ_{i,q_i}$ intersecting $\Gamma_i$.
We require the neighborhoods $\UU_i$ and $\ZZ_i$ to be sufficiently $C^2$-small so that, for each $i_1,i_2,k,l$ with $i_1\neq i_2$, the intersection numbers $\#(\nu_{i_1}\cap\nu_{i_2})$ are independent of the specific choices of $\nu_{i_1}\in\UU_{i_1}$ and $\nu_{i_2}\in\UU_{i_2}$, or of the specific choice of $\nu_{i_1}\in\ZZ_{i_1,k}$ and $\nu_{i_2}\in\ZZ_{i_2,l}$. By Lemma~\ref{l:intertwining}, there exist $t_1,t_2,t_3\in[0,\infty)$ such that, for each $i\in\{1,...,n\}$, there exists $\nu_i\in\UU_i$ and $k_i\in\{1,...,q_i\}$ satisfying
\begin{align*}
\phi_h^{t_1}(\nu_i)\in\ZZ_{i,k_i},
\qquad
\phi_g^{t_3}\circ\phi_h^{t_2+t_1}(\nu_i)\in\UU_i.
\end{align*}
For each $i\in\{1,...,n\}$, we fix a closed geodesic $\zeta_i\in \ZZ_{i,k_i}\cap\Gamma_i$.
We claim that $\bm\zeta=(\zeta_{1},...,\zeta_{n})$ has the same flat link type as $\bm\gamma=(\gamma_1,...,\gamma_n)$. Indeed, for each $i\neq j$ the components $\gamma_i$ and $\gamma_j$ have distinct flat knot type or lengths satisfying $|\ell_i-\ell_j|>2\epsilon$, and therefore the components of $\bm\zeta$ are pairwise distinct. We consider the continuous path of multi-loops $\bm\nu_t=(\nu_{1,t},...,\nu_{n,t})$, where
\begin{align*}
\nu_{i,t}
:=
\left\{
  \begin{array}{@{}ll}
    \phi_h^t(\nu_i), & t\in[0,t_1+t_2], \vspace{5pt} \\ 
    \phi_g^{t-t_1-t_2}\circ\phi_h^{t_2+t_1}(\nu_i), &  t\in[t_1+t_2,t_1+t_2+t_3]. 
  \end{array}
\right.
\end{align*}
By \cite[Lemma~3.3]{Angenent:2005aa}, the number of intersections between two geometrically distinct curves evolving for time $t$ under the curve shortening flow is a non-increasing function of $t$. Therefore, for all $i_1\neq i_2$ the functions $t\mapsto\#(\nu_{i_1,t}\cap\nu_{i_2,t})$ are non-increasing, and therefore constant, since
\[
\#(\nu_{i_1,0}\cap\nu_{i_2,0})
=
\#(\gamma_{i_1}\cap\gamma_{i_2})
=
\#(\nu_{i_1,t_1+t_2+t_3}\cap\nu_{i_2,t_1+t_2+t_3}).
\]
This implies that each $\bm\nu_t$ has the same flat link type of $\bm\gamma$. Finally, $\bm\nu_{t_1}$ has the same flat link type as $\bm\zeta$.
\end{proof}

\section{Forced existence of closed geodesics}
\label{s:forcing}

As we already mentioned in Section \ref{ss:global}, Angenent's work \cite{Angenent:2005aa} implies that the local homology of a primitive relative flat knot type $\KK$ is independent of the choice of the admissible  Riemannian metric $g$. 
Its non-vanishing implies the existence of closed geodesics of flat knot type $\KK$ for any such $g$. With this in mind, in order to prove Theorems~\ref{mt:multiplicity} and~\ref{mt:sphere}, we need to study the local homology of certain flat knot types relative to a contractible simple closed geodesic on Riemannian closed oriented surfaces of positive genus, or relative to a pair of disjoint simple closed geodesics on Riemannian 2-spheres. We shall conveniently employ certain model Riemannian metrics, introduced by Donnay \cite{Donnay:1988ab}, and Burns and Gerber \cite{Burns:1989aa}, for which enough properties of the geodesic flow are known.

\subsection{Focusing caps}
The main ingredient for the construction of the model Riemannian metrics is a certain Riemannian disk of revolution $(B^2,g_0)$, introduced by Donnay, Burns and Gerber in the above mentioned works, and whose construction we now recall.
As a set, the disk $B^2=B^2(r_0)\subset\R^2$ is the compact Euclidean one of radius $r_0>0$ centered at the origin. The Riemannian metric is of the form
\[g_0 = dr^2+\rho(r)^2d\theta^2,\] 
where $(r,\theta)\in[0,r_0]\times\R$ are the polar coordinates on $B^2$, and $\rho:[0,r_0]\to[0,1]$ is a smooth function such that $\rho(0)=0$ and $\rho(r_0)=1$. The associated Gaussian curvature is independent of $\theta$, and indeed is given by 
\[R_{g_0}(r)=-\ddot\rho(r)/\rho(r).\]
The Riemannian disk $(B^2,g_0)$ is called a \emph{focusing cap} when it satisfies the following three properties.
\begin{itemize}
\setlength{\itemsep}{5pt}

\item[(i)] $\dot\rho|_{[0,r_0)}>0$,

\item[(ii)] $\dot R_{g_0}|_{(0,r_0)}<0$ and $R_{g_0}(r_0)=0$,

\item[(iii)] The equator $\partial B^2=\{r=r_0\}$ is a closed geodesic.

\end{itemize}

\begin{Remark}\label{r:attach_cap}
One can also require that all the derivatives of the function $\rho$ vanish at $r=r_0$. With this extra assumption, the focusing cap can be smoothly attached to a flat cylinder $([r_0,r_1]\times\R/2\pi\Z,dr^2+d\theta^2)$.
\end{Remark}

Let $\pi:SB^2\to B^2$ be the unit tangent bundle of the focusing cap. Any unit vector $v\in SB^2$ based at a point $x=\pi(v)$ distinct from the origin is uniquely determined by the triple $(r,\theta,\xi)\in[0,r_0]\times\R\times\R$, where $(r,\theta)$ are the polar coordinates of $x$, and $\xi$ is the signed angle between the parallel through $x$ oriented counterclockwise and $v$ (Figure~\ref{f:total_rotation}(a)). We denote by $\gamma_v(t)$ the associated geodesic such that $\gamma(0)=x$ and $\dot\gamma(0)=v$.

We recall that orthogonal Jacobi fields $J$ along a geodesic $\gamma$ have the form $J(t)=u(t)n_\gamma(t)$, where $n_\gamma$ is a unit normal vector field to $\dot\gamma$, and the real-valued function $u$ is a solution of the scalar Jacobi equation $\ddot u+R_{g_0} u=0$. We briefly call such a $u$ a scalar orthogonal Jacobi field along $\gamma$.

The following proposition, due to Donnay, Burns, and Gerber, summarizes the main properties of the focusing caps.

\begin{Prop}
\label{p:focusing_cap}$ $
\begin{itemize}
\setlength{\itemsep}{5pt}

\item[(i)] \cite[Prop.~3.1]{Donnay:1988ab} The focusing cap is non-trapping: the only geodesic $\gamma(t)$ defined for all $t\geq0$ is the equator $\partial B^2$.

\item[(ii)] \cite[Sect.~5]{Donnay:1988ab} Let $v=(r_0,\theta_0,\xi)\in \partial SB^2$ be a unit vector based at a point $x\in\partial B^2$ on the equator and pointing transversely inside the focusing cap with a tangent angle $\xi\in(0,\pi/2]$. The associated geodesic $\gamma_v(t)=(r(t),\theta(t))$ reaches the equator at some positive time $t=T$,  spanning total rotation
\begin{align*}
 \Theta(\xi):=
 \left\{
  \begin{array}{@{}ll}
    \pi, & \mbox{if }\xi=\pi/2, \vspace{5pt}\\ 
    \theta(T)-\theta(0), & \mbox{if }\xi\in(0,\pi/2), \\ 
  \end{array} 
 \right.
\end{align*}
see Figure~\ref{f:total_rotation}(b).
The function $\Theta:(0,\pi/2]\to\R$ is smooth, and satisfies 
\[ \dot\Theta|_{(0,\pi/2]}<0.\]

\begin{figure}
\includegraphics{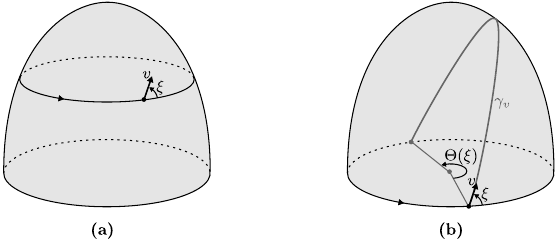}
\caption{\textbf{(a)} The coordinate $\xi$. \textbf{(b)} The total rotation $\Theta(\xi)$.}
\label{f:total_rotation}
\end{figure}

\item[(iii)] \cite[Prop.~6.1]{Donnay:1988ab} Let $v=(r_0,\theta_0,\xi)\in \partial SB^2$ be a unit tangent vector based at some point of the equator $\partial B^2$ and pointing transversely inside the focusing cap with a tangent angle $\xi\in(0,\pi/2]$. Let $\gamma_v:[0,T]\to B^2$ be the associated maximal geodesic segment. Any scalar orthogonal Jacobi field $u:[0,T]\to\R$ along $\gamma_v$ satisfies 
\begin{align*}
u(T) & = -u(0)+\sin(\xi)\dot\Theta(\xi)\dot u(0),\\
\dot u(T) & = -\dot u(0).
\end{align*}
By symmetry, if instead $\xi\in[\pi/2,\pi)$ such a scalar Jacobi field satisfies 
\begin{align*}
u(T) & = -u(0)+\sin(\xi)\dot\Theta(\pi-\xi)\dot u(0),\\
\dot u(T) & = -\dot u(0).
\end{align*}

\end{itemize}
\end{Prop}

\subsection{Model metric in positive genus}
We now introduce the model Riemannian metric $g_0$ on a closed connected oriented surface of positive genus $M$. We first introduce some notation, and refer the reader to, e.g., \cite[Sect.~1.11]{Guillarmou:2024aa}, for the background on the geometry of unit tangent bundles of Riemannian surfaces.
Let $\psi_t:SM\to SM$ be the geodesic flow on the unit tangent bundle associated with $g_0$. Its orbits have the form $\psi_t(\dot\gamma(0))=\dot\gamma(t)$, where $\gamma:\R\to M$ is a geodesic parametrized with unit speed $\|\dot\gamma\|_{g_0}\equiv1$. The unit tangent bundle $SM$ admits a frame $X,X_\perp,V$ that is orthonormal with respect to the Sasaki Riemannian metric on $SM$ induced by $g_0$, where $X$ is the geodesic vector field, $V$ is a unit vector field tangent to the fibers of $SM$, and $X_\perp=[X,V]$. The sub-bundle of $T(SM)$ spanned by $X_\perp,V$ is the contact distribution of $SM$, and is invariant under the linearized geodesic flow $d\psi_t$.

\begin{Prop}[\cite{Donnay:1988ab,Burns:1989aa}]
\label{p:model_metric_positive_genus}
On any closed connected oriented surface $M$ of positive genus, there exists a Riemannian metric $g_0$ with the following properties:
\begin{itemize}
\setlength{\itemsep}{5pt}

\item[(i)] It has a contractible simple closed geodesic $\zeta$, which bounds an open disk $B\subset M$.

\item[(ii)] The Gaussian curvature $R_{g_0}$ is strictly negative on $U:=M\setminus\overline B$, and vanishes along $\zeta$.

\item[(iii)] The disk $B$ is non-trapping: no forward orbit $\psi_{[0,\infty)}(v)$ is entirely contained in the subset $SB\subset SM$.

\item[(iv)] The cone bundle $C$ over $SU$, given by
\begin{align}
\label{e:cone_field}
C_v
=
\big\{
aX_\perp(v) + b V(v)\ 
\big|\ 
a,b\in\R\mbox{ such that }ab\leq0
\big\},
\end{align}
is positively invariant and contracted by the linearized geodesic flow: for all $v\in SU$ and $t>0$ such that $\psi_t(v)\in SU$, we have
\begin{align*}
 d\psi_t(v)C_v\setminus\{0\}\subset \interior(C_{\psi_t(v)});
\end{align*}
see Figure~\ref{f:cone}.
\end{itemize}
\end{Prop}

\begin{figure}
\begin{footnotesize}
\includegraphics{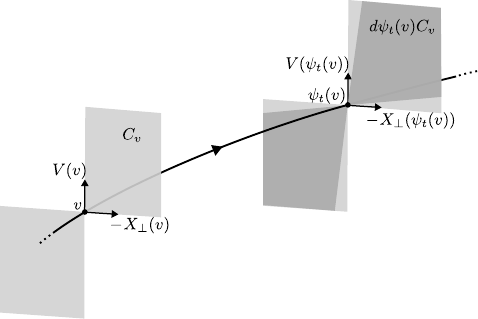}
\end{footnotesize}
\caption{The cone bundle $C$.}
\label{f:cone}
\end{figure}

The Riemannian metric $g_0$ is constructed starting from a hyperbolic metric on the pinched surface $M\setminus\{x\}$, having a cusp at $x$. Next one modifies the metric in a neighborhood $N$ of $x$, so that $N$ contains a simple closed geodesic $\zeta\subset N$ that bounds a focusing cap $\overline B$, the Gaussian curvature vanishes along $\zeta$ and is strictly negative outside $\overline B$ (see \cite[Section~1]{Burns:1989aa}). Property (iii)  follows from Proposition~\ref{p:focusing_cap}(i). The cone invariance in property (iv) can be easily verified along orbit segments that stay in the complement of the focusing cap $\overline B$, where the curvature in negative. The fact that the invariance still holds after crossing the interior $B$ of the focusing cap follows from Proposition~\ref{p:focusing_cap}(iii).

We recall that a closed geodesic $\gamma$, parametrized with unit speed and having minimal period $\tau>0$, is \emph{hyperbolic} when $d\psi_\tau(\dot\gamma(0))$ has an eigenvalue $q\in\R\setminus[-1,1]$, and thus its eigenvalues are $1,q,1/q$. The unstable bundle $\Eu$ over $\dot\gamma$ is the line bundle given by
\[
\Eu_{\dot\gamma(0)}:=\ker\big(d\psi_\tau(\dot\gamma(0))-qI\big).
\]
The eigenvalue $q$ is called the unstable Floquet multiplier of $\gamma$.

\begin{Lemma}\label{l:hyperbolic}
Let $\gamma$ be a closed geodesic of $g_0$ geometrically distinct from $\zeta$. Then $\gamma$ is hyperbolic, and $\Eu_{\dot\gamma(t)}\subset C_{\dot\gamma(t)}$ for all $t\in\R$ such that $\gamma(t)\in U$ $($here, $\gamma$ is parametrized with unit speed $\|\dot\gamma\|_{g_0}\equiv1$$)$.
\end{Lemma}

\begin{proof}
Let $\gamma$ be a closed geodesic of $g_0$ geometrically distinct from $\zeta$. Proposition~\ref{p:model_metric_positive_genus}(iii) implies that $\gamma$ must intersect the open set $U$ where the cone bundle $C$ is defined. 
We parametrize $\gamma$ with unit speed, with $\gamma(0)$ in $U$, and denote by $\tau>0$ its minimal period.
Proposition~\ref{p:model_metric_positive_genus}(iv) implies that $d\psi_\tau(\dot\gamma(0))$ is a contraction on the space of lines $\ell\subset C_{\dot\gamma(0)}$ (here, line means 1-dimensional vector subspace). Therefore, $d\psi_\tau(\dot\gamma(0))|_{C_{\dot\gamma(0)}}$ has a unique fixed line $\ell=d\psi_\tau(\dot\gamma(0))\ell\subset C_{\dot\gamma(0)}$, which must be an eigenspace of  $d\psi_\tau(\dot\gamma(0))$ corresponding to an eigenvalue $q\in\R\setminus[-1,1]$.
\end{proof}

The parity of the Morse index of a hyperbolic closed geodesics is completely determined by its Floquet multipliers. On a Riemannian surface, a hyperbolic closed geodesic $\gamma$ with unstable Floquet multiplier $q$ has $\ind(\gamma)$ even if and only if $q>0$, namely if and only if the unstable bundle $\Eu_{\dot\gamma}$ is orientable, see e.g. \cite[Corollary~3.6]{Wilking:2001aa}.

\begin{Lemma}\label{l:same_parity}
Let $\KK$ be a flat-knot type relative to $\zeta$, and $\gamma_0,\gamma_1\in\Gamma_{g_0}(\KK)$ two closed geodesics. Then the Morse indices $\ind(\gamma_0)$ and $\ind(\gamma_1)$ have the same parity.
\end{Lemma}

\begin{proof}
Let $n_\zeta$ be the unit normal vector field to $\zeta$ pointing outside $B$.
The boundary of $SB$ splits as a disjoint union $\partial SB=\dot\zeta\cup-\dot\zeta\cup\partial_+SB\cup\partial_-SB$, where
\begin{align*}
 \partial_\pm SB
 =
 \big\{ v\in\partial SB\ \big|\ \pm g(v,n_\zeta)>0 \big\}.
\end{align*}
We extend the cone field $C$ to $SM\setminus(\dot\zeta\cup-\dot\zeta)$ as follows: first we extend it continuously to $\partial_\pm SB $ as in~\eqref{e:cone_field}; next, for each $v\in\partial_-SB$ and $t>0$ such that $\psi_{(0,t]}(v)\subset SB$, we set
\begin{align*}
 C_{\psi_t(v)} = d\psi_t(v)C_v.
\end{align*}
The resulting cone field $C$ on $SM\setminus(\dot\zeta\cup-\dot\zeta)$ is discontinuous at $\partial_+SB$, but nevertheless it is continuous (and even piecewise smooth) elsewhere. Moreover, Proposition~\ref{p:model_metric_positive_genus}(iv) guarantees that $C$ has a semi-continuity with respect to the Hausdorff topology, and
\begin{align*}
d\psi_t(v)C_v\subseteq C_{\psi_t(v)},
\qquad \forall v\in SM\setminus(\dot\zeta\cup-\dot\zeta),\ t>0.
\end{align*}

Let $\gamma_s\in\KK$ be an isotopy from $\gamma_0$ to $\gamma_1$ within the relative flat knot type $\KK$. We fix parametrizations $\gamma_s:S^1\looparrowright M$ depending smoothly on $s$, and define a continuous map
\[ \Gamma:[0,1]\times S^1\to SM,\qquad \Gamma(s,t)=\dot\gamma_s(t)/\|\dot\gamma_s(t)\|_{g_0}. \]
We also write $\Gamma_s(t):=\Gamma(s,t)$.
An orientation on the cone bundle $\Gamma^*C$ is a choice of connected component of $C_{\Gamma(s,t)}\setminus\{0\}$ which is continuous in $(s,t)$. Notice that this notion makes sense even if $C$ is only semi-continuous. For each $s\in\{0,1\}$, the unstable bundle $\Eu_{\dot\gamma_s}$ is contained in $\Gamma_s^*C$. Therefore  $\Eu_{\dot\gamma_s}$ is orientable if and only if $\Gamma_s^*C$ is orientable, and thus if and only if the whole $\Gamma^*C$ is orientable. We conclude that $\Eu_{\dot\gamma_0}$ and $\Eu_{\dot\gamma_1}$ are either both orientable or both unorientable. 
\end{proof}

In order to detect homologically visible flat knot types relative to $\zeta$, we first study the closed geodesics in the negatively curved open subset $U=M\setminus\overline B$. Since $\overline U$ is a compact surface with geodesic boundary, it is preserved by the curve shortening flow of $g_0$, meaning that the evolution of any immersed loop starting inside $\overline U$ remains in $\overline U$. The same holds for more classical gradient flows, for instance for the one in the setting of piecewise broken geodesics (Section~\ref{ss:energy}), and allows us to apply Morse theoretic methods to the subspace of loops contained in $\overline U$.

While the statement of Theorem~\ref{mt:multiplicity} involves free homotopy classes of loops in $M$, in the next lemma we rather consider free homotopy classes of loops in $\overline U$, that is, connected components of $C^\infty(S^1,\overline U)$.

\begin{Lemma}\label{l:unique_minimizer}
In any connected component $\UU\subset C^\infty(S^1,\overline U)$ consisting of loops that are non-contractible in $M$, there exists a unique closed geodesic $\gamma$ of $g_0$, and such a $\gamma$ is the shortest loop in $\UU$, i.e.
\begin{align*}
 L_{g_0}(\gamma) = \min_{\eta\in\UU} L_{g_0}(\eta).
\end{align*}
\end{Lemma}

\begin{proof}
Let $\UU\subset C^\infty(S^1,\overline U)$ be a connected component of loops that are non-contractible in $M$. We fix an element $\gamma_0\in\UU$ that is an immersed loop and has minimal number of self-intersections. In particular, there is no non-empty subinterval $(a,b)\subset S^1$ such that $\gamma_0(a)=\gamma_0(b)$ and $\gamma|_{[a,b]}$ is a contractible loop in $\overline U$. Since $\overline U$ has geodesic boundary $\partial U=\zeta$, the curve shortening flow $\phi^t$ of $g_0$ preserves the compact set $\overline U$. Namely, the immersed loop $\gamma_t:=\phi^t(\gamma_0)$ is contained in $\overline U$ for all $t\in[0,t_{\gamma_0})$. Since $\gamma_0$ is non-contractible in $M$ and has no subloops that are contractible in $\overline U$, Lemma~\ref{l:subloops}(i) implies that $\gamma_t$ converges to a closed geodesic $\gamma\in\UU$ as $t\to t_{\gamma_0}$. 
Since $\gamma$ is non-contractible in $M$, it is geometrically distinct from $\zeta$, and therefore it is contained in the open set $U$.
Since the Gaussian curvature $R_{g_0}$ is negative on $U$, all the closed geodesics in $\UU$ are strict local minimizers of the length functional $L_{g_0}$.  
If $\UU$ contained two geometrically distinct closed geodesics $\alpha,\beta$, we could define the  min-max value
\begin{align*}
 c:=\inf_{h_s} \max_{s\in[0,1]} L_{g_0}(h_s),
\end{align*}
where the infimum ranges over the family of homotopies $h_s\in\UU$ such that $h_0=\alpha$ and $h_1=\beta$. Standard Morse theory would imply that the value $c$ is the length of a closed geodesic in $\UU$ that is not a strict local minimizer of $L_{g_0}$, which would give a contradiction.
\end{proof}

\begin{Lemma}\label{l:infinitely_many_Donnay}
In any connected component $\UU\subset C^\infty(S^1,M)$ of non-contractible loops, there exists a sequence of hyperbolic closed geodesics $\gamma_n\in\UU$ of $g_0$ with diverging length $L_{g_0}(\gamma_n)\to\infty$, Morse index $\ind(\gamma_n)=1$, and such that $\gamma_n\cap\zeta\neq\varnothing$.
\end{Lemma}

\begin{proof}
Consider an arbitrary connected component of non-contractible loops $\UU\subset C^\infty(S^1,M)$. We can find a loop $\alpha\in\UU$ that is fully contained in $\overline U$ and has starting point $\alpha(0)=\zeta(0)$. Since the closed geodesic $\zeta$ is contractible, for each positive integer $n$ the concatenation $\alpha*\zeta^n$ is again a loop in $\UU$. Let $\UU_n$ be the connected component of $C^\infty(S^1,\overline U)$ containing $\alpha*\zeta^n$ (Figure~\ref{f:connected_components}). 
\begin{figure}
\begin{footnotesize}
\includegraphics{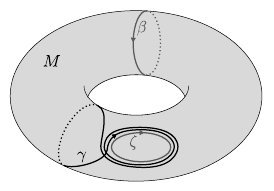}
\end{footnotesize}
\caption{A loop $\gamma\in\UU_2$. Here, $\UU$ is the connected component of the loop $\beta$.}
\label{f:connected_components}
\end{figure}%
The $\UU_n$'s are pairwise distinct, and contain loops that are non-contractible. By Lemma~\ref{l:unique_minimizer}, $\UU_n$ contains a unique closed geodesic $\alpha_n:S^1\to\overline U$, which is the shortest loop in $\UU_n$. 
Since $\{\alpha_n\ |\ n\geq1\}$ is a discrete non-compact subset of $C^\infty(S^1,\overline U)$, while the space of closed geodesics in $C^\infty(S^1,\overline U)$ of length bounded from above by any given constant is compact, we infer that $L_{g_0}(\alpha_n)\to\infty$  as $n\to\infty$.

We consider the min-max values
\begin{align*}
 c_n:=\inf_{h_s} \max_{s\in[0,1]} L_{g_0}(k_s),
\end{align*}
where the infimum ranges over the family of continuous homotopies $h_s\in\UU$, $s\in[0,1]$, such that $h_0=\alpha_1$ and $h_1=\alpha_n$. Since the circles $\{\alpha_n(t+\cdot)\ |\ t\in S^1\}\subset\UU$ are strict local minimizers of the length functional $L_{g_0}$, Morse theory implies that $c_n=L_{g_0}(\gamma_n)>L_{g_0}(\alpha_n)$, where $\gamma_n$ is a closed geodesic of 1-dimensional min-max type. Since $\gamma_n$ belongs to $\UU$, it is geometrically distinct from $\zeta$. By Lemma~\ref{l:hyperbolic}, $\gamma_n$ is hyperbolic, and in particular non-degenerate. Therefore, $\gamma_n$ has Morse index $\ind(\gamma_n)=1$. 
\end{proof}

\begin{proof}[Proof of Theorem~\ref{mt:multiplicity}]
Let $g$ be a Riemannian metric on $M$ having a contractible simple closed geodesic $\zeta$, and $\UU\subset C^\infty(S^1,M)$ a primitive free homotopy class of loops. 
In particular, $\UU$ does not contain contractible loops, and therefore does not contain $\zeta$ nor any of its iterates.
Let $g_0$ be the Riemannian metric on $M$ given by Proposition~\ref{p:model_metric_positive_genus}, having the same $\zeta$ as simple closed geodesic. By Lemma~\ref{l:infinitely_many_Donnay}, $g_0$ admits an infinite sequence of hyperbolic closed geodesics $\gamma_n\in\UU$ such that $\ind(\gamma_n)=1$, $L(\gamma_n)\to\infty$, and $\gamma_n\cap\zeta\neq\varnothing$. Since $\UU$ is primitive, none of the $\gamma_n$'s is an iterated closed geodesic, and therefore we can assume that the $\gamma_n$'s are pairwise geometrically distinct. Each $\gamma_n$ has some primitive flat knot type $\KK_n$ relative to $\zeta$. Notice that any $\gamma\in\KK_n$ must intersect $\zeta$. By Lemma~\ref{l:same_parity}, all closed geodesics of $g_0$ in $\KK_n$ must have odd Morse index. By Lemma~\ref{l:local_homology_parity}, we have
\[
C_*(\KK_n) \cong \bigoplus_{\gamma\in\Gamma_{g_0}(\KK_n)} C_*(\gamma),
\] 
and in particular $C_1(\KK_n)$ contains a subgroup isomorphic to $C_1(\gamma_n)\cong\Z$. This, together with Proposition~\ref{p:Morse_inequality}, implies that $\KK_n$ contains a primitive closed geodesic of the original Riemannian metric $g$. We have two possible cases:
\begin{itemize}

\item If the family $\KK_n$, $n\geq1$, consists of infinitely many pairwise distinct flat knot types, then we immediately conclude that $\cup_{n\geq1}\KK_n$ contains infinitely many primitive closed geodesics of $g$.

\item If there exists a sequence of positive integers $n_j\to\infty$ such that 
\[\KK:=\KK_{n_1}=\KK_{n_2}=\KK_{n_3}=...\] 
then $C_1(\KK)$ contains a subgroup isomorphic to 
\[\bigoplus_{j\geq1}C_1(\gamma_{n_j})=\bigoplus_{j\geq1}\Z.\] 
In particular $C_1(\KK)$ has infinite rank.
If the space of closed geodesics $\Gamma_g(\KK)$ is not discrete, in particular it contains infinitely many closed geodesics. If instead $\Gamma_g(\KK)$ is discrete, by Proposition~\ref{p:Morse_inequality} we infer
\begin{align*}
 \sum_{\gamma\in\Gamma_g(\KK)}
 \rank(C_1(\gamma))
 \geq
 \rank(C_1(\KK))
 =\infty,
\end{align*}
and since each local homology group $C_1(\KK)$ has finite rank (Lemma~\ref{l:non_deg_visible}), we infer that $\KK$ contains infinitely many closed geodesics.
\qedhere
\end{itemize}
\end{proof}

\subsection{Model metric in genus zero}
We now construct a model Riemannian metric $g_0$ on the 2-sphere $S^2$.
Let $B_1$ and $B_2$ be two copies of a focusing cap as in Remark~\ref{r:attach_cap},  with rotation function $\Theta:(0,\pi)\to\R$. We consider the flat cylinder \[(C=[0,1]\times S^1,dr^2+d\theta^2),\] 
where $r$ is the coordinate on $[0,1]$, and $\theta$ is the coordinate on $S^1$.
We obtain a 2-sphere of revolution $(S^2,g_0)$ by capping off $C$ with $B_1$ and $B_2$.
Such a 2-sphere has a family of equatorial simple closed geodesics $\gamma_z:=\{z-1\}\times S^1$, for $z\in[1,2]$. The extremal ones $\gamma_1$ and $\gamma_{2}$ are the boundaries of the focusing caps $B_1$ and $B_2$ respectively. There is also an $S^1$ family of simple closed geodesics consisting of the meridians, that is, the geodesics passing through the center $x_1$ of $B_1$, and thus passing through the center $x_2$ of $B_2$ as well (Figure~\ref{f:model_sphere}).

\begin{Lemma}\label{l:only_meridians}
Any geodesic other than the $\gamma_z$'s and the meridians has a transverse self-intersection.
\end{Lemma}

\begin{proof}
Since the focusing caps are non-trapping (Proposition~\ref{p:focusing_cap}(i)), any geodesic other than the $\gamma_z$'s must enter both $B_1$ and $B_2$. Proposition~\ref{p:focusing_cap}(ii) implies that all geodesic segments entering a focusing cap with tangent angle $\xi\in(0,\pi/2)$ with respect to the boundary of the cap must exit after spanning a total rotation angle $\Theta(\xi)>\pi$. 

Let $\zeta:\R\to S^2$ be a geodesic distinct from the $\gamma_z$'s and the meridians. In particular, $\zeta$ does not go through the centers $x_1$ and $x_2$ of the focusing caps, since only the meridians do so. Every time $\zeta$ enters $B_1$, it does so with the same signed tangent angle $\xi$ with respect to $\partial B_1$. 
By the $S^1$-symmetry of the sphere of revolution $(S^2,g)$, it is enough to consider the case $\xi\in(0,\pi/2)$.
We parametrize $\zeta$ such that $\zeta(0)$ belongs to $B_1$ and $\dot\zeta(0)$ is tangent to a parallel of the cap. Let $\theta(t)$ be the global angle coordinate along $\zeta(t)$, and let us assume that $\theta(0)=0$ without loss of generality, so that $\theta(t)=-\theta(-t)$ for all $t>0$. Notice that, by our assumption on $\xi$, $\theta(t)$ is monotone increasing. Let $\tau_1>0$ be the minimal positive number such that $\zeta(t)$ exits $B_1$ at time $t=\tau_1$. By symmetry, $\zeta$ entered $B_1$ at time $t=-\tau_1$, and spans inside $B_1$ a total rotation 
\[\theta(\tau_1)-\theta(-\tau_1)=2\theta(\tau_1)=\Theta(\xi).\]
Let $\tau_2>\tau_1$ be the smallest positive number such that $\zeta(t)$ enters $B_2$ at time $t=\tau_2$, and $\tau_3>\tau_2$  the minimal time such that $\dot\zeta(\tau_3)$ is tangent to a parallel of the focusing cap $B_2$. By the north-south symmetry of the sphere of revolution $(S_2,g)$, we have $\tau_3=\tau_1+\tau_2$, and $\theta(\tau_3)-\theta(\tau_2)=\theta(\tau_1)$. Depending on the value of the tangent angle $\xi$, the geodesic $\zeta$ can be open or closed, and in this latter case its minimal period must be larger than or equal to $2\tau_3$. In both cases, the points $\zeta(-\tau_3)$ and $\zeta(\tau_3)$ lie on the same parallel of $B_2$, and the total rotation of $\zeta|_{[-\tau_3,\tau_3]}$ is bounded from below as
\begin{align*}
\theta(\tau_3)-\theta(-\tau_3)
>
4\theta(\tau_1)
=
2\Theta(\xi)
>2\pi.
\end{align*}
This implies that $\zeta|_{[-\tau_3,\tau_3]}$ must have a transverse self-intersection.
\end{proof}

\begin{figure}
\includegraphics{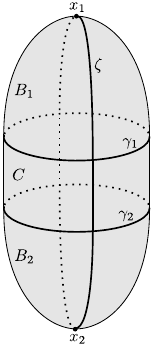}
\caption{The model sphere, the simple closed geodesics $\gamma_1,\gamma_2$ at the boundary of the focusing caps, and a meridian $\zeta$.}
\label{f:model_sphere}
\end{figure}

Let $\ell$ be the length of the meridians of $(S^2,g_0)$, and consider the angle coordinate $\theta$ on the focusing cap $B_1$. Each value of $\theta\in \R/2\pi\Z$ determines the unique meridian 
\[\zeta_\theta:\R/\ell\Z\hookrightarrow S^2\] parametrized with unit speed $\|\dot\zeta_\theta\|_{g_0}\equiv1$, going out of the center $\zeta_\theta(0)=x_1$ of the focusing cap along the half-meridian of angle $\theta$. There is an evident non-trivial $\ell$-periodic orthogonal Jacobi field along $\zeta_\theta$, given by $\partial_\theta\zeta_\theta$. 

\begin{Lemma}\label{l:no_orthogonal_Jacobi}
Any orthogonal $\ell$-periodic Jacobi field along $\zeta_\theta$ is of the form $\lambda\,\partial_\theta\zeta_\theta$ for some $\lambda\in\R$.
\end{Lemma}

\begin{proof}
Let $a_1<b_1<a_2<b_2<a_1+\ell$ be time values such that $\zeta_\theta|_{[a_i,b_i]}\subset B_i$ and $\zeta_\theta(a_i),\zeta_\theta(b_i)\in\partial B_i$. Notice that $\zeta_\theta$ travels on the flat cylinder $C$ in the intervals $[b_1,a_2]$ and $[b_2,a_1+\ell]$, and therefore $a_2-b_1=a_1+\ell-b_2=1$. Let $J(t)=u(t)\,n_{\zeta_\theta}(t)$ be an orthogonal Jacobi field along $\zeta_\theta$, where $u$ is the associated scalar Jacobi field. 
By Proposition~\ref{p:focusing_cap}(iii), we have
\begin{equation}
\label{e:Jacobi_meridian_1}
\begin{split}
u(b_i) & = -u(a_i)+\dot u(a_i)\,\dot\Theta(\pi/2),
\\
\dot u(b_i) & = -\dot u(a_i).
\end{split}
\end{equation}
On a flat Riemannian surface, scalar orthogonal Jacobi fields have constant derivative. Therefore there exists $c\in\R$ such that
$\dot u|_{[b_2,a_1+\ell]}\equiv - \dot u|_{[b_1,a_2]} \equiv c$,
and we have
\begin{align}
\label{e:Jacobi_meridian_2}
 u(a_1+\ell)-u(b_2)
 =
 -\big(u(a_2)-u(b_1)\big)
 =
 c.
\end{align}
By~\eqref{e:Jacobi_meridian_1} and \eqref{e:Jacobi_meridian_2}, we infer
\begin{align*}
 u(a_1+\ell)
 & = u(b_2)+c
 = -u(a_2) + \dot u(a_2)\,\dot\Theta(\pi/2) +c \\
 & = -u(b_1) + \dot u(a_2)\,\dot\Theta(\pi/2) + 2c \\
 & = u(a_1) - \dot u(a_1)\,\dot\Theta(\pi/2) + \dot u(a_2)\,\dot\Theta(\pi/2) + 2c\\
 & = u(a_1) -2c\,\dot\Theta(\pi/2) + 2c.
\end{align*}
Assume now that the orthogonal Jacobi field $J$ is $\ell$-periodic, so that $u$ is $\ell$-periodic as well. Since $\dot\Theta(\pi/2)<0$ by Proposition~\ref{p:focusing_cap}(ii),  the previous identity implies that $c=0$. In particular, $u$ is constant on the interval $[a_2,b_1]$. Since the scalar Jacobi field associated to $\partial_\theta\zeta_\theta$ is constant and non-zero on $[a_2,b_1]$ as well, we conclude that $J\equiv\lambda \partial_\theta\zeta_\theta$ for some $\lambda\in\R$.
\end{proof}

Let $\gamma_1$, $\gamma_2$ be two disjoint embedded circles in the 2-sphere $S^2$, forming the flat link $\bm\gamma=(\gamma_1,\gamma_2)$. We denote by $\KK(\bm\gamma)$ the flat knot type relative to $\bm\gamma$ consisting of those embedded loops intersecting each $\gamma_i$ in two points. Notice that $\KK(\bm\gamma)$ is primitive, and therefore we can apply to it the Morse theoretic techniques of Section~\ref{s:Morse}. 

\begin{Lemma}\label{l:K_gamma1_gamma2}
The local homology of $\KK(\bm\gamma)$ is given by
\begin{align*}
C_d(\KK(\bm\gamma))
\cong
\left\{
  \begin{array}{@{}cc}
    \Z, & d\in\{n,n+1\}, \vspace{5pt}\\ 
    0, & d\not\in\{n,n+1\}, 
  \end{array}
\right. 
\end{align*}
for some integer $n\geq1$.
\end{Lemma}

\begin{proof}
The local homology of $\KK(\bm\gamma)$ is independent of the choice of Riemannian metric on $S^2$ having $\bm\gamma$ as flat link of closed geodesics. Therefore, we conveniently choose the model Riemannian metric of revolution $g_0$ introduced above, with $\gamma_1:=\partial B_1$ and $\gamma_2:=\partial B_2$ being the pair of equatorial closed geodesics at the boundary of the focusing caps $B_1$ and $B_2$ respectively (Figure~\ref{f:model_sphere}). We fix a parametrization $\gamma_1:S^1\hookrightarrow S^2$ with constant speed.

For each $s\in S^1$, we denote by $\zeta_s:S^1\hookrightarrow S^2$ the unique meridian (which is a simple closed geodesic) parametrized with constant speed and such that $\zeta_s(0)=\gamma_1(s)$ and $\dot\zeta_s(0)$ points south inside the flat cylinder $C$. Each $\zeta_s$ has relative flat knot type $\KK(\bm\gamma)$.
The space
\[
Z:=\big\{\zeta_s\ |\ s\in S^1\big\}
\]
is a critical circle of the energy functional $E$. By Lemma~\ref{l:only_meridians}, the critical torus $S^1\cdot Z$ is a connected component of $\crit(E)$. By Lemma~\ref{l:no_orthogonal_Jacobi}, any 1-periodic Jacobi field along $\zeta_s$ has the form $c_1\partial_s\zeta_s+c_2\dot\zeta_s$ for some $c_1,c_2\in\R$. Namely, the space of such 1-periodic Jacobi fields is precisely the tangent space $T_{\zeta_s}(S^1\cdot Z)$. Since the 1-periodic Jacobi fields along a closed geodesic in $\crit(E)$ span the kernel of the Hessian of the energy, we infer
\begin{align*}
 \ker(d^2E(\zeta_s)) = T_{\zeta_s}(S^1\cdot Z),
\end{align*}
that is, $S^1\cdot Z$ is a non-degenerate critical manifold of $E$. The symmetry of the 2-sphere of revolution $(S^2,g_0)$ implies that the negative bundle $N\to Z$ is orientable. Therefore, seeing $Z$ as a space of unparametrized oriented closed geodesics in $\Gamma_{g_0}(\KK)$, Lemma~\ref{l:non_deg_visible} implies that $Z$ has local homology
\begin{align*}
C_*(Z) 
\cong
\left\{
  \begin{array}{@{}cc}
    \Z, & d\in\{\ind(Z),\ind(Z)+1\}, \vspace{5pt}\\ 
    0, & d\not\in\{\ind(Z),\ind(Z)+1\}.
  \end{array}
\right.
\end{align*}
Finally, since the space of closed geodesics $\Gamma_{g_0}(\KK(\bm\gamma))$ is precisely $Z$ according to Lemma~\ref{l:only_meridians}, we conclude
\[
C_*(\KK(\bm\gamma))\cong C_*(Z).
\qedhere
\]
\end{proof}

\begin{proof}[Proof of Theorem~\ref{mt:sphere}]
Let $(S^2,g)$ be a Riemannian 2-sphere having two disjoint simple closed geodesics $\gamma_1,\gamma_2$, and consider the flat link $\bm\gamma=(\gamma_1,\gamma_2)$. By Lemma~\ref{l:K_gamma1_gamma2}, the primitive relative flat knot type $\KK(\bm\gamma)$ has non-trivial local homology. Therefore, the Morse inequality of Proposition~\ref{p:Morse_inequality} implies that $\KK(\bm\gamma)$ contains at least a closed geodesic $\gamma$, which is therefore a simple closed geodesic intersecting each $\gamma_i$ in two points.
\end{proof}

\section{Birkhoff sections}
\label{s:Birkhoff_sections}

\subsection{From closed geodesics to Birkhoff sections}

Let $(M,g)$ be a closed orientable Riemannian surface, and $\psi_t:SM$ its geodesic flow. For each open subset $W\subset SM$, the associated \emph{trapped set} is defined as
\begin{align*}
\trap(W) := \Big\{v\in SM\ \Big|\ \psi_t(v)\in W\mbox{ for all $t>0$ large enough} \Big\}.
\end{align*}
By a \emph{convex geodesic polygon}, we mean an open ball $B\subset M$ whose boundary is piecewise geodesic with at least one corner and all inner angles at its corners are less than $\pi$. We stress that the closure $\overline B$ is not required to be an embedded compact ball. Typical examples of convex geodesic polygons are the simply connected components of the complement of a finite collection of closed geodesics, as in Lemma~\ref{l:std_collection} below.

The main result of this section is the following.

\begin{Thm}\label{t:complete_system}
Any closed connected orientable Riemannian surface $(M,g)$ admits a finite collection of closed geodesics $\gamma_1,...,\gamma_n$ whose complement $U:=M\setminus(\gamma_1\cup...\cup\gamma_n)$ satisfies $\trap(SU)=\varnothing$, and each connected component of $U$ is a convex geodesic polygon.
\end{Thm}

In the terminology of \cite[Section~4.2]{Contreras:2022ab}, the family $\gamma_1,...,\gamma_n$ provided by Theorem~\ref{t:complete_system} is a ``complete system of closed geodesics with empty limit subcollection''. Postponing the proof of Theorem~\ref{t:complete_system} to the next subsection, we first derive the proof of Theorem~\ref{mt:Birkhoff_sections}.

\begin{proof}[Proof of Theorem~\ref{mt:Birkhoff_sections}]
Let $\gamma_1,...,\gamma_n$ be the finite collection of closed geodesics provided by Theorem~\ref{t:complete_system}, so that the complement $U:=M\setminus(\gamma_1\cup...\cup\gamma_n)$ satisfies $\trap(SU)=\varnothing$, and each connected component of $U$ is a convex geodesic polygon. We fix a unit-speed parametrization $\gamma_i:\R/L(\gamma_i)\Z\to M$, and consider a normal vector field $n_{\gamma_i}$. Each $\gamma_i$ has two associated Birkhoff annuli $A_i^+$ and $A_i^-$, defined as
\begin{align*}
A_i^\pm:=\big\{ v\in S_{\gamma_i(t)}M\ \big|\ t\in L(\gamma_i)\Z,\ \pm g(n_{\gamma_i}(t),v)\geq0\big\}.
\end{align*}
Namely, $A_i^\pm$ is an immersed compact annulus in $SM$ with boundary $\partial A_i^\pm=\dot\gamma_i\cup-\dot\gamma_i$, and whose interior consists of those unit tangent vectors $v\in SM$ based at some point $\gamma_i(t)$ and pointing to the same side of $\gamma_i$ as $\pm n_{\gamma_i}(t)$. Notice that $A_i^\pm$ is an \emph{immersed} surface of section: namely, it is almost a surface of section, the only missing property being the embeddedness of $\interior(A_i^\pm)$ into $SM\setminus\partial A_i^\pm$. The union
\begin{align*}
\Upsilon:=\bigcup_{i=1,...,n} \Big( A_i^+ \cup A_i^- \Big)
\end{align*}
is an immersed surface of section as well.

Notice that $SU=SM\setminus\Upsilon$. Since $\trap(SU)=\varnothing$, for each $v\in SM$ there exists a minimal $\tau_v>0$ such that $\psi_{\tau_v}(v)\in\Upsilon$. Since each connected component of $U$ is a convex geodesic polygon, there exists a neighborhood $W\subset M$ of $\partial U=\gamma_1\cup...\cup\gamma_n$ and $T>0$ such that any geodesic segment $\gamma:[-T,T]\to M$ parametrized with unit speed and such that $\gamma(0)\in W$ cannot be fully contained in $U$. This readily implies that the hitting time $\tau_v$ is uniformly bounded from above by a constant $\tau>0$ for all $v\in SM$. A surgery procedure due to Fried \cite{Fried:1983aa}, also described in the geodesics setting in \cite[Section~4.1]{Contreras:2022ab}, allows to resolve the self-intersections of $\Upsilon$, and produce a surface of section $\Sigma$ with the same boundary as $\Upsilon$, contained in an arbitrarily small neighborhood of $\Upsilon$, and such that for each $v\in\Upsilon$ the orbit segment $\psi_{(-\tau,\tau)}(z)$ intersect $\Sigma$. This implies that, for each $v\in SM$, the orbit segment $\psi_{(0,2\tau)}$ intersects $\Sigma$. Therefore $\Sigma$ is a Birkhoff section.
\end{proof}

\begin{Remark}
Since the Birkhoff section $\Sigma$ provided by Theorem~\ref{mt:Birkhoff_sections} is constructed by applying Fried surgery to the collection of pairs of Birkhoff annuli of a finite family of closed geodesics, one can show that the first return time 
\[\tau:\interior(\Sigma)\to(0,\infty),\qquad \tau(v)=\min\big\{t\in(0,T]\ \big|\ \psi_t(z)\in\Sigma \big\} \] 
is bounded from below by a positive constant.
\end{Remark}

\subsection{Complete system of closed geodesics}

The following result, which is a special case of \cite[Theorem 3.4]{Contreras:2022ab}, provides the main criterium to produce an open set $U$ with the properties asserted in Theorem~\ref{t:complete_system}.

\begin{Thm}[Contreras, Knieper, Mazzucchelli, Schulz]\label{t:non_trapping}
Let $(M,g)$ be a closed orientable Riemannian surface. If a convex geodesic polygon $B\subset M$ does not contain simple closed geodesics, then $\trap(SB)=\varnothing$.
\hfill\qed
\end{Thm}

In order to apply this result in the proof of Theorem~\ref{t:complete_system} we need to detect enough closed geodesics. As a starting point, on surfaces of positive genus, we employ the following standard family of closed geodesics, which exists for any Riemannian metric.

\begin{figure}
\begin{footnotesize}
\includegraphics{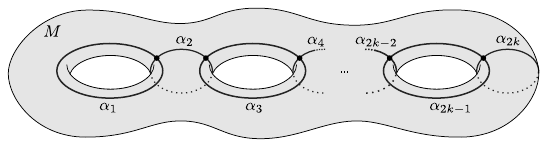}
\end{footnotesize}
\caption{The collection of closed geodesics of Lemma~\ref{l:std_collection}.}
\label{f:std_collection}
\end{figure}

\begin{Lemma}
[\cite{Contreras:2022ab}, Lemma 4.5]
\label{l:std_collection}
On any closed connected orientable Riemannian surface $(M,g)$ of genus $k\geq1$, there exist non-contractible simple closed geodesics $\alpha_1,...,\alpha_{2k}$ as depicted in Figure~\ref{f:std_collection}. Namely,  two such $\alpha_i$ and $\alpha_j$ intersect transversely in a single point if $|i-j|=1$, they are disjoint if $|i-j|\geq2$, and the complement $M\setminus(\alpha_1\cup...\cup\alpha_{2k})$ is simply connected.
\hfill\qed
\end{Lemma}

On the 2-sphere, we have a similar statement.

\begin{Lemma}\label{l:std_collection_S2}
On any Riemannian 2-sphere, there exist two simple closed geodesics $\alpha_1,\alpha_2$ intersecting each other transversely and non-trivially.
\end{Lemma}

\begin{proof}
By Lusternik-Schnirelmann theorem \cite{Lusternik:1929wa,De-Philippis:2022aa}, any Riemannian 2-sphere admits three geometrically distinct simple closed geodesics. If two of them are disjoint, Theorem~\ref{mt:sphere} provides a simple closed geodesic intersecting both.
\end{proof}

Let $(M,g)$ be a closed connected oriented Riemannian surface, and $\alpha_1,...,\alpha_{2k}$ the family of simple closed geodesics provided by Lemmas~\ref{l:std_collection} and~\ref{l:std_collection_S2}. We denote their union, seen as a connected subset of $M$, by 
\[A:=\alpha_1\cup...\cup\alpha_{2k}.\]
Any connected component of $M\setminus A$ is a convex geodesic polygon, and may contain simple closed geodesics. In order to apply Theorem~\ref{t:non_trapping} we need to add sufficiently many of them, as well as more closed geodesics provided by Theorems~\ref{mt:multiplicity} and~\ref{mt:sphere}, to our initial collection $A$. We shall need one last statement borrowed from \cite{Contreras:2022ab}.

\begin{Lemma}
[\cite{Contreras:2022ab}, Lemma 3.6]\label{l:annuli}
On any closed orientable Riemannian surface $(M,g)$, there there exists a constant $a>0$ with the following property: for any embedded compact annulus $N\subset M$ with $\area(N,g)\leq a$ and whose boundary is the disjoint union of two simple closed geodesics $\gamma_1\cup\gamma_2$, we have 
\[
\tag*{\qed}
\tfrac29L(\gamma_2)\leq L(\gamma_1)\leq\tfrac92 L(\gamma_2).
\]
\end{Lemma}

Any simple closed geodesic $\gamma$ contained in $M\setminus A$ bounds a unique open disk $B_\gamma\subset M\setminus A$. Gauss-Bonnet theorem guarantees that such a disk cannot be too small: the Gaussian curvature $R_g:M\to\R$ must attain positive values somewhere in $B_\gamma$, and the area of $B_\gamma$ is bounded from below as
\begin{align}
\label{e:Gauss_Bonnet_area_bound}
 \area(B_\gamma,g)\geq\frac{2\pi}{\max(R_g)}.
\end{align}
We denote by $\GG_\gamma$ the family of contractible simple closed geodesics $\zeta$ contained in the open disk $B_\gamma$. The following is the last ingredient for the proof of Theorem~\ref{t:complete_system}.

\begin{Lemma}\label{l:innermost}
For each contractible simple closed geodesic $\gamma\subset M\setminus A$ such that $\GG_\gamma\neq\varnothing$, there exists $\zeta\in\GG_\gamma$ such that $\GG_\zeta=\varnothing$.
\end{Lemma}

\begin{proof}
We set
\begin{align*}
b_\gamma:=\inf_{\zeta\in\GG_\gamma} \area(B_\zeta,g),
\end{align*}
which is a positive value according to~\eqref{e:Gauss_Bonnet_area_bound}.
We fix $\zeta\in\GG_\gamma$ such that 
\[
b_\gamma \leq\area(B_{\zeta},g)\leq b_\gamma +a,
\]
where $a>0$ is the constant given by Lemma~\ref{l:annuli}. This implies that every $\eta\in\GG_{\zeta}$ has length $L(\eta)\leq\tfrac92 L(\zeta)$. Therefore $\GG_{\zeta}$, seen as a subspace of $C^\infty(S^1,M)$ endowed with the $C^\infty$ topology, is compact. Consider another sequence $\eta_m\in\GG_{\zeta}$, for $m\geq1$, such that $\area(B_{\eta_m},g)\to b_{\zeta}$ as $m\to\infty$. By compactness, up to extracting a subsequence we have that $\eta_m$ converges in the $C^\infty$ topology to some $\eta\in\GG_{\zeta}$ such that $\area(B_\eta,g)=b_{\zeta}$. This implies that $\GG_\eta=\varnothing$.
\end{proof}

\begin{proof}[Proof of Theorem~\ref{t:complete_system}]
Since every simple closed geodesic contained in $M\setminus A$ bounds a disk of area uniformly bounded from below as in~\eqref{e:Gauss_Bonnet_area_bound}, Lemma~\ref{l:innermost} implies that there exists a maximal finite collection of pairwise disjoint contractible simple closed geodesics $\zeta_1,...,\zeta_h$ contained in $M\setminus A$ and such that $\GG_{\zeta_i}=\varnothing$ for all $i\in\{1,...,h\}$. We denote by $Z:=\zeta_1\cup...\cup\zeta_h$ their disjoint union. Here, ``maximal'' means that, for any other contractible simple closed geodesic $\zeta$ contained in $M\setminus (A \cup Z)$, the associated $\GG_{\zeta}$ must contain some $\zeta_i$.
If $Z=\varnothing$, then $A$ is the desired collection of closed geodesics: indeed, $U:=M\setminus A$ is a convex geodesic polygon that does not contain any simple closed geodesic, and Theorem~\ref{t:non_trapping} implies that $\trap(SU)=\varnothing$.

\begin{figure}
\begin{footnotesize}
\includegraphics{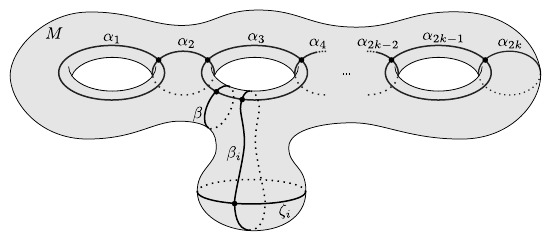}
\end{footnotesize}
\caption{The loop $\beta$, and the closed geodesic $\beta_i$ in the same free homotopy class, intersecting the contractible simple closed geodesic $\zeta_i$.}
\label{f:innermost}
\end{figure}

We now consider the case $Z\neq\varnothing$.
We need to argue differently for the surfaces of positive genus and for the 2-sphere.
\begin{itemize}
\setlength{\itemsep}{5pt}

\item If $M$ has positive genus, let $\beta$ be an embedded loop that intersects a unique $\alpha_j$, and such intersection is transverse and consists of a single point. Theorem~\ref{mt:multiplicity} implies that, for each $i\in\{1,...,h\}$, there exists a  closed geodesic $\beta_i$ in the same free homotopy class of loops of $\beta$ and such that $\beta_i\cap\zeta_i\neq\varnothing$ (Figure~\ref{f:innermost}). 
Notice that $\beta_i$ must intersect $\alpha_j$ too, since the intersection between $\beta$ and $\alpha_j$ is homologically essential. 

\item If $M=S^2$, Theorem~\ref{mt:sphere} implies that, for each $i\in\{1,...,h\}$, there exist a closed geodesic $\beta_i$ intersecting both $\zeta_i$ and $A$.

\end{itemize}

We set $B:=\beta_1\cup...\cup\beta_h$.
Since every connected component of $M\setminus A$ is simply connected and $A\cup Z\cup B$ is path-connected, every connected component of the complement $U:=M\setminus (A\cup Z\cup B)$ is simply connected, and thus is a convex geodesic polygon. Notice that $U$ does not contain any simple closed geodesic; indeed, if it contained a simple closed geodesic $\eta$, the maximality of the collection $Z$ would imply that $\GG_\eta$ contained some $\zeta_i$, contradicting the path-connectedness of $A\cup Z\cup B$. As before, Theorem~\ref{t:non_trapping} implies that $\trap(SU)=\varnothing$.
\end{proof}

\begin{Remark}
\label{r:bound_boundary_components}
The argument in the proof allows to bound from above the number $n$ of closed geodesics $\gamma_1,...,\gamma_n$ provided by Theorem~\ref{t:complete_system}. Indeed, the area lower bound~\eqref{e:Gauss_Bonnet_area_bound} readily implies that the number $h$ of closed geodesics in $Z$ is bounded from above as
\begin{align*}
 h \leq \frac{\area(M,g)\max(R_g)}{2\pi}.
\end{align*}
The set $B$ also contains $h$ closed geodesics, whereas $A$ contains exactly \[2k=2\max\big\{1,\mathrm{genus}(M)\big\}\] closed geodesics. Therefore
\begin{align*}
n
=
2k+2h
\leq 2\max\big\{1,\mathrm{genus}(M)\big\}+\frac{1}{\pi}\,\area(M,g)\max(R_g).
\end{align*}
The Birkhoff section $\Sigma\looparrowright SM$ of Theorem~\ref{mt:Birkhoff_sections} is built from the union of the Birkhoff annuli of the closed geodesics $\gamma_1,...,\gamma_n$, and therefore $2n$ Birkhoff annuli which all together have $4n$ boundary components. The surgery procedure that constructs $\Sigma$ out of the Birkhoff annuli does not increase the number of boundary components. Therefore $\Sigma$ has $b=4n$ boundary components, with
\begin{align*}
 b
 \leq
 8\max\big\{1,\mathrm{genus}(M)\big\}+\frac{4}{\pi}\,\area(M,g)\max(R_g).
\end{align*}
\end{Remark}

\bibliography{_biblio}

\vspace{10pt}

\end{document}